\documentclass[11pt]{article}

\usepackage{amssymb,amsmath, amsfonts, amsthm}

\usepackage{graphicx}
\usepackage{pgfplots}
\usepackage{tikz}
\usepackage{caption}
\usepackage{booktabs}
\usepackage{array}
\usepackage{verbatim}
\usepackage{subfig}
\usepackage{geometry}
\usepackage{fancyhdr}

\newcommand{\vertex}{\node[vertex]}
\tikzstyle{vertex}=[circle, draw, inner sep=0pt, minimum size=6pt]
\newtheorem{theorem}{Theorem}
\newtheorem{lemma}{Lemma}
\newtheorem{corollary}{Corollary}
\newtheorem{obs}{Observation}

\newcommand{\smallqed}{{\tiny ($\Box$)}}
\newcommand{\dist}{\hskip2pt$\rm{dist}$}
\newcommand{\NewSarah}[1]{\textcolor{purple}{\bf #1}}

\begin{document}

\title{Characterizing all nonbipartite well-edge-dominated graphs}

\author{Sarah E. Anderson$^a$ \and Kirsti Kuenzel$^b$
}

\maketitle

\begin{center}
$^a$ Department of Mathematics, University of St. Thomas, St. Paul, Minnesota, USA\\
$^b$ Department of Mathematics, Trinity College, Hartford, CT, USA\\

\end{center}
\medskip

\maketitle
\begin{abstract}
Given a graph $G$, a set $F$ of edges is an edge dominating set of $G$ if every edge in $G$ is either in $F$ or adjacent to an edge in $F$. A graph $G$ is said to be well-edge-dominated if every minimal edge dominating set has the same cardinality. This definition is the edge version of domination in that a set $D\subseteq V(G)$ is a dominating set if every vertex in $G$ is in $D$ or adjacent to a vertex in $D$ and the domination number $\gamma(G)$ is the minimum cardinality among all dominating sets. In this paper, we complete the characterization of all nonbipartite, well-edge-dominated graphs. In addition, we produce an infinite class of graphs that satisfy the well-known Vizing's conjecture in domination theory that states $\gamma(G\Box H) \ge \gamma(G)\gamma(H)$ where $G\Box H$ is the Cartesian product of $G$ and $H$.
\end{abstract}

\maketitle
\section{Introduction}
In any graph $G$, an independent set of edges $F \subseteq E(G)$ (i.e. no pair of edges share a common vertex) is referred to as \emph{matching} in $G$. The maximum cardinality among all matchings in $G$ is the matching number and denoted by $\alpha'(G)$. On the other hand, a set of edges $F \subseteq E(G)$ is an \emph{edge dominating set} of $G$ if each edge in $G$ is either in $F$ or adjacent to an edge in $F$. The minimum cardinality among all edge dominating sets of $G$ is the edge domination number which we denote by $\gamma'(G)$. The matching number and edge domination number are the edge counterparts to the independence number $\alpha(G)$ (maximum vertex independent set) and domination number $\gamma(G)$ (minimum vertex dominating set). There is a key relationship between matchings and edge dominating sets in any graph in that every maximal matching is necessarily a minimal edge dominating set, i.e. $\gamma'(G) \le \alpha'(G)$. Lewin \cite{L-1974} and Meng \cite{M-1974} independently define a graph $G$ to be \emph{equimatchable} if every maximal matching in $G$ has cardinality $\alpha'(G)$. A polynomial time algorithm for recognizing when a graph is equimatchable was provided by Lesk, Plummer, and Pulleyblank \cite{LPP-1984}. Ideally, we would be able to define a structural characterization for all equimatchable graphs. However, to date this problem has been too difficult. There has been great progress in characterizing certain subclasses of equmatchable graphs. For instance, Fendrup, Hartnell, and Vestergaard \cite{FHV-2010} proved that a connected equimatchable graph with no cycles of length less than $5$ is either $C_5$, $C_7$, or belongs to the family of bipartite graphs that contains $K_2$ as well as those graphs $G$ with a bipartition $A\cup B$ where $A$ is the set of support vertices in $G$ (vertices adjacent to a vertex with degree one). They also proved that any equimatchable graph with no 3- or 4-cycles is \emph{well-edge-dominated} \cite{fhn-1988}, meaning that every minimal edge dominating set in the graph has cardinality $\gamma'(G)$. Continuing to focus on the length of the smallest cycle, B\"{u}y\"{u}k\c{c}olak et al. \cite{BGO-2020} provided a structural characterization of all nonbipartite equimatchable graphs containing $4$-cycles. Additionally, Akbari et al. \cite{AGHI-2018} studied $r$-regular equimatchable graphs and proved that  if $r$ is even, then $G \in \{K_{r+1}, K_{r,r}\}$ and if $r$ is odd and $G$ is triangle-free, then $G \in \{C_5, C_7, K_{r,r}\}$. For additional results in equimatchable graphs, see \cite{DE-2019, EK-2016, KPS-2003}.

Note that if $G$ is well-edge-dominated, then any maximal matching in $G$ must also have cardinality $\gamma'(G)$, meaning that $G$ is equimatchable. Therefore, the class of well-edge-dominated graphs is a subset of the class of all equimatchable graphs. As with equimatchable graphs, one would prefer to find a structural description of the class of all well-edge-dominated graphs. Again focusing on the length of the shortest cycle in the graph, it was shown by Anderson et al. \cite{AKR-2022} that the only connected, triangle-free,  nonbipartite well-edge-dominated graphs are $C_5, C_7$, and what they referred to as $C_7^*$, which is obtained from $C_7$ by adding a chord between two vertices at distance 3. Berg et al. \cite{BCKKPRV-2025} went on to characterize all connected well-edge-dominated graphs containing exactly one triangle, followed by a characterization of all connected nonbipartite $K_4$-free well-edge-dominated graphs provided by Anderson and Kuenzel \cite{AK-2026}. In this paper, we complete the characterization of all connected nonbipartite well-edge-dominated graphs by showing that every connected well-edge-dominated graph containing an induced $K_4$ is obtained from a particular operation on a smaller connected well-edge-dominated graph. 

To address the question of why one would like to have a structural characterization of all well-edge-dominated graphs, we point out that if we take the line graph $\mathcal{L}(G)$ of a well-edge-dominated graph $G$, then every minimal vertex dominating set of $\mathcal{L}(G)$ has the same cardinality, and we say that $\mathcal{L}(G)$ is well-dominated. An interesting property that we will encounter is that some  well-edge-dominated graphs from our characterization are ``minimal" in the sense that the removal of any edge from the graph decreases the edge domination number. For such a graph $G$, this implies that if we take the line graph $\mathcal{L}(G)$ of $G$, then for any $v \in V(\mathcal{L}(G))$,  $\gamma(\mathcal{L}(G)-v) = \gamma(\mathcal{L}(G))-1$. That is, every vertex in $\mathcal{L}(G)$ is  vertex critical with respect to domination. Characterizing vertex critical graphs is known to be very difficult. Therefore,  taking the line graph of such  well-edge-dominated graphs yields new examples of vertex critical graphs which can then be studied. Furthermore, there is a well-known conjecture in domination theory known as Vizing's conjecture that states that for any pair of graphs $G$ and $H$, $\gamma(G\Box H) \ge \gamma(G)\gamma(H)$, where $G\Box H$ denotes the Cartesian product of $G$ and $H$. (see \cite{VCsurvey}).  Indeed, we show in this paper that if $\mathcal{L}(G)$ is a line graph of a well-edge-dominated graph, then $\mathcal{L}(G)$ is not a counterexample to Vizing's conjecture. 

The remainder of this paper is organized as follows. In Section 1.1, we provide necessary definitions, terminology and previous results, including the description of the class of the connected nonbipartite $K_4$-free well-edge-dominated graphs that will be needed throughout this paper. In Section 2, we provide the structural description of our class $\mathcal{G^*}$  and prove that $G$ is connected, nonbipartite, and well-edge-dominated if and only if $G\in \mathcal{G^*}$ or $G$ is one of thirteen exceptions. In Section 3, we show that for every $G \in \mathcal{G^*}$ that the line graph of $G$, denoted $\mathcal{L}(G)$, satisfies $\gamma(\mathcal{L}(G)\Box H) \ge \gamma(\mathcal{L}(G))\gamma(H)$ for any graph $H$, implying that it satisfies Vizing's conjecture. In Section 4, we provide concluding remarks and open problems. 

\subsection{Definitions and Terminology}
Throughout this paper, we study only simple, finite graphs $V = (V(G), E(G))$ where we write $n(G) = |V(G)|$. Given any $v \in V(G)$, the {\it{open neighborhood of $v$}}, denoted by $N_G(v)$ (or $N(v)$ if the context is clear) is the set of all vertices adjacent to $v$. The {\it{closed neighborhood of $v$}} is defined to be $N_G[v] = N_G(v) \cup \{v\}$, or simply $N[v]$ when the context is clear. Given a set $S\subseteq V(G)$, we define the {\it{open neighborhood of $S$}} analogously i.e.  $N_G(S) = \cup_{v\in S}N_G(v)$. The {\it{degree}} of a vertex $v$ is $\deg_G(v) = |N_G(v)|$. Any vertex with degree one is referred to as a {\it{leaf}} and its lone neighbor in the graph is a {\it{support vertex}} (or specifically the support vertex of $v$ in $G$).  Any edge incident to a leaf is referred to as a \emph{pendent edge}. The \emph{girth} of $G$, denoted $g(G)$, is defined to be the length of the smallest cycle in $G$. We define $\Delta(G) = \max_{v\in V(G)}\{\deg_G(v)\}$. We also have edge versions of these definitions. Given any edge $f \in E(G)$, we denote the {\it{open edge neighbor of $f$}} as $N_e(f)$ and define it to be the set of all edges in $G$ adjacent to $f$. The {\it{closed edge neighborhood of $f$}} is analogously defined as $N_e[f] = N_e(f) \cup \{f\}$ and for any set $F\subseteq E(G)$, the {\it{open neighborhood of $F$}} is $N_e(F) = \cup_{f \in F}N_e(f)$. 

We say that $F\subseteq E(G)$ is an {\it{edge dominating set of $G$}} if every edge in $G$ is either in $F$ or adjacent to an edge in $F$. The minimum cardinality among all edge dominating sets of $G$ is the {\it{edge domination number of $G$}} and denoted by $\gamma'(G)$ (sometimes denoted $\gamma_e(G)$ in other papers). A set $M\subseteq E(G)$ is a matching in $G$ if it is independent, or no pair of edges in $M$ share a common vertex. The maximum cardinality among all matchings in $G$ is the matching number and is denoted $\alpha'(G)$. There is a well-known chain of inequalities relating these two parameters:
\[\gamma'(G)\le \alpha'(G)\le \Gamma'(G)\]
where $\Gamma'(G)$ is the maximum cardinality among all minimal edge dominating sets in $G$. $G$ is said to be {\it{equimatchable}} if all maximal independent sets have cardinality $\alpha'(G)$ and $G$ is said to be {\it{well-edge-dominated}} if all minimal edge dominating sets have cardinality $\gamma'(G)$. Therefore, if $G$ is well-edge-dominated, then $\gamma'(G) = \alpha'(G) = \Gamma'(G)$ meaning that $G$ is equimatchable. 

Similarly, we can define the above in terms of vertices. We say a set $D \subseteq V(G)$ is a {\it{dominating set of $G$}} if every vertex in $G$ is either in $D$ or adjacent to a vertex in $D$. The minimum cardinality among all dominating sets in $G$ is the {\it{domination number of $G$}} and is denoted by $\gamma(G)$. A set $I\subseteq V(G)$ is an {\it{independent set in $G$}} if no pair of vertices in $I$ are adjacent in $G$ and $\alpha(G)$ is the maximum cardinality among all independent sets in $G$, referred to as the independence number. The corresponding chain of inequalitiies with respect to vertices is 
\[\gamma(G)\le \alpha(G)\le \Gamma(G)\]
where $\Gamma(G)$ is the maximum cardinality among all minimal dominating sets of $G$. We say that $G$ is {\it{well-dominated}} if every minimal dominating set of $G$ has cardinality $\gamma(G)$ and $G$ is {\it{well-covered}} if all maximal independent sets have cardinality $\alpha(G)$. Therefore, if $G$ is well-dominated, then it is necessarily well-covered. 

Given any two vertices $u,v\in V(G)$, the graph obtained by {\it{identifying $u$ and $v$}} is found by removing $u$ and $v$ from $G$ and replacing them with a vertex $w$ where $w$ is adjacent to each vertex in $(N_G(u) \cup N_G(v))- \{u,v\}$. Given a set $S\subset V(G)$, we will refer to the operation of {\it{identifying each vertex in $S\subset V(G)$ into a single vertex}} to mean removing all vertices in $S$ from $G$ and replacing with a vertex $w$ where $w$ is adjacent to each vertex in $N_G(S) - S$. When we are dealing with a bipartite graph, we will write $G = (A\cup B, E)$ where we mean $V(G) = A\cup B$ such that $A$ and $B$ are independent sets. Recall that given a graph $G$, the {\it{line graph of $G$}} is denoted by $\mathcal{L}(G)$ and is the graph defined by $V(\mathcal{L}(G)) = \{v_e: e \in E(G)\}$ and $v_e, v_f\in V(\mathcal{L}(G))$ are adjacent if and only if $e$ and $f$ are adjacent in $G$. The \emph{complement} of $G$, denote $\overline{G}$, is the graph with $V(\overline{G}) = V(G)$ and $u,v \in V(\overline{G})$ are adjacent if and only if $uv\not\in E(G)$. 

The \emph{clique number} of $G$, denoted $\omega(G)$, is the cardinality of the largest induced clique in $G$. A partition of $V(G) = \Pi_1 \cup \cdots \cup \Pi_k$ is referred to as a \emph{clique partition} if $G[\Pi_j]$ is a clique in $G$ for each $j \in [k]$. The minimum $k$ for which $V(G) = \Pi_1 \cup \cdots \cup \Pi_k$ is a clique partition of $G$ is the \emph{clique partition number} of $G$, denoted $\theta(G)$. A map $f: V(G) \to \{1, \dots, k\}$ is a proper coloring of $G$ provided $f(u) \ne f(v)$ whenever $uv \in E(G)$, and we say that $G$ is $k$-colorable. The minimum value of $k$ for which $G$ is $k$-colorable is the \emph{chromatic number} of $G$ and is denoted $\chi(G)$. It is known that $\omega(G) \le \chi(G)$. Recall that we refer to $G$ as a \textit{perfect graph} provided $\omega(H) = \chi(H)$ for every induced subgraph $H$ of $G$. Note that we know $\alpha(G) = \omega(\overline{G})$ and $\chi(G) = \theta(\overline{G})$. Recall that Lov\'a{s}z proved the Perfect Graph Theorem, originally conjectured by Claude Berg.

\begin{theorem}\cite{L-1972} A graph $G$ is perfect if and only if $\overline{G}$ is perfect.
\end{theorem}

As a result, if $G$ is a perfect graph, then $\overline{G}$ is also perfect and we know $\omega(\overline{G})= \chi(\overline{G})$. This in turn implies that $\alpha(G) = \theta(G)$. Also note that it is known that line graphs of bipartite graphs are perfect. To see this, it was shown by K\"{o}nig for a bipartite graph $G$ that the edge chromatic number, denoted $\chi'(G)$, satisfies $\chi'(G) = \Delta(G)$. Therefore, $\chi(\mathcal{L}(G)) = \chi'(G)$ and $\omega(\mathcal{L}(G)) = \Delta(G)$. Thus, for any induced subgraph $H$ of $G$, $H$ is also bipartite and therefore $\chi(\mathcal{L}(H)) = \omega(\mathcal{L}(H))$. We state this as an observation for ease of reference.

\begin{obs}\label{obs:bipartite} The line graph of a bipartite graph is perfect.
\end{obs}

  \section{Previous Results on Well-edge-dominated Graphs}
  Bipartite, equimatchable graphs were studied by B\"{u}y\"{u}k\c{c}olak et al. in \cite{BGO-2023}. Notably, they proved the following.
  
  \begin{lemma}\cite{BGO-2023} Let $G = (A\cup B, E)$ with $|A| \le |B|$ be a connected equimatchable bipartite graph. Then each vertex $u\in A$ satisfies at least one of the following:
  \begin{enumerate}
  \item[(i)] $u$ is a support vertex in $G$, 
  \item[(ii)] $u$ is included in a subgraph $K_{2,2}$ in $G$.
  \end{enumerate}
  \end{lemma}
  
  It was also shown in \cite{BGO-2023} that if $G= (A\cup B, E)$ is a bipartite equimatchable graph with $|A|\le |B|$, then by greedily choosing edges in $G$ to create a maximal matching, we will have chosen a set $M$ that saturates $A$. Since every well-edge-dominated graph is equimatchable, we have the following. 
  
\begin{lemma}\cite{BGO-2023}\label{lem:satA} Let $G = (A \cup B, E)$ be a well-edge-dominated bipartite graph where $|A| \le |B|$. There exists a maximal matching in $G$ that saturates $A$.
\end{lemma}

There is also a nice property that all equimatchable and well-edge-dominated graphs enjoy.

\begin{lemma}\label{lem:reduce}\cite{AKR-2022} Let $M$ be a matching in a graph $G$. If $G$ is well-edge-dominated, then $G - N_e[M]$ is well-edge-dominated. If $G$ is equimatchable, then $G - N_e[M]$ is equimatchable. 
\end{lemma}

As mentioned above, the impetus for continuing to characterize well-edge-dominated graphs based on the girth of $G$, was the following result. 
  
  \begin{theorem}\label{thm:girth5red}\cite{AKR-2022} If $G$ is well-edge-dominated with $g(G) \ge 4$, then either $G$ is bipartite or $G \in \{C_5, C_7, C_7^*\}$.
  \end{theorem}
  
Then in \cite{BCKKPRV-2025}, a characterization of well-edge-dominated graphs containing exactly one triangle was provided, where $\mathcal{T}$ and $\mathcal{F}$ are infinite classes of well-edge-dominated graphs.

 \begin{theorem}\cite{BCKKPRV-2025}\label{thm:onetriangle} $G$ is a well-edge-dominated graph with exactly one triangle if and only if $G \in \mathcal{T}\cup \mathcal{F} \cup \{K_3, Cr, \mathcal{H}, \mathcal{DH}\}$.
 \end{theorem}

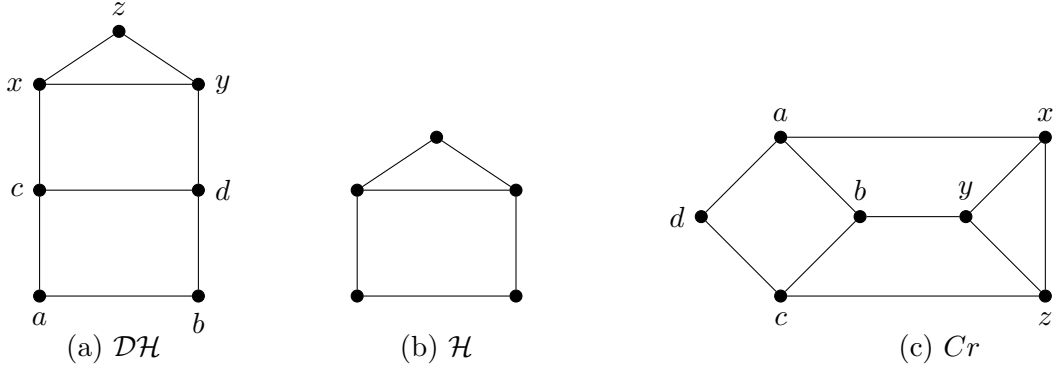
\begin{figure}[h!]
\begin{center}
\begin{tikzpicture}[scale=.7]
    \vertex (1) at (0,0)  [fill, scale=.75, label=below:$a$]{};
    \vertex (2) at (0,2) [fill, scale=.75, label=left:$c$]{};
    \vertex (3) at (3,0)  [fill, scale=.75, label=below:$b$]{};
    \vertex (4) at (3,2)   [fill, scale=.75, label=right:$d$]{};
    \vertex (5) at (0,4) [fill, scale=.75, label=left:$x$]{};
    \vertex (6) at (3,4)  [fill, scale=.75, label=right:$y$]{};
    \vertex (7) at (1.5, 5)  [fill, scale=.75, label=above:$z$]{};

    \vertex (8) at (6,0) [fill, scale=.75]{};
    \vertex (9) at (9,0)  [fill, scale=.75]{};
    \vertex (10) at (6,2)  [fill, scale=.75]{};
    \vertex (11) at (9,2)  [fill, scale=.75]{};
    \vertex (12) at (7.5,3)  [fill, scale=.75]{};
    
     \vertex (1a) at (14,3)  [scale=.75, fill=black, label=above:$a$]{};
    \vertex (2a) at (12.5,1.5)  [scale=.75, fill=black, label=left:$d$]{};
    \vertex (3a) at (14,0)  [scale=.75, fill=black, label=below:$c$]{};
    \vertex (4a) at (15.5,1.5)  [scale=.75, fill=black, label=above:$b$]{};
    \vertex (5a) at (17.5,1.5)  [scale=.75, fill=black, label=above:$y$]{};
    \vertex (6a) at (19,3)  [scale=.75, fill=black, label=above:$x$]{};
    \vertex (7a) at (19,0)  [scale=.75, fill=black, label=below:$z$]{};
    \node(A) at (1.4, -1)[]{(a) $\mathcal{DH}$};
    \node(B) at (7.5, -1)[]{(b) $\mathcal{H}$};
    \node(C) at (17, -1)[]{(c) $Cr$};

    \path 
    (1) edge (2)
    (1) edge (3)
    (4) edge (2)
    (4) edge (3)
    (5) edge (2)
    (6) edge (4)
    (7) edge (6)
    (7) edge (5)
    (5) edge (6)

    (8) edge (9)
    (8) edge (10)
    (11) edge (10)
    (11) edge (9)
    (11) edge (12)
    (12) edge (10)
    
       (1a) edge (2a)
    (2a) edge (3a)
    (3a) edge (4a)
    (4a) edge (1a)
    (4a) edge (5a)
    (5a) edge (6a)
    (5a) edge (7a)
    (6a) edge (7a)
    (6a) edge (1a)
    (7a) edge (3a)
    ;

\end{tikzpicture}
\caption{The dream house (a), the house graph (b) and the crystal graph (c)}
\label{fig:houses}
\end{center}
\end{figure}

The authors provided a complete characterization of all $K_4$-free well-edge-dominated graphs with girth $3$ in \cite{AK-2026} based on the infinite classes of well-edge-dominated graphs containing exactly one triangle.  It relied on a class of graphs referred to as  $\mathcal{P}, \mathcal{W}$, and $\mathcal{G}$ and defined as follows. First note that the \emph{diamond graph} is the graph obtained from the complete graph $K_4$ by removing an edge. Given a diamond $D$ with vertices $\{u, v, w, x\}$ such that $ wx \not\in E(D)$, we refer to $u$ and $v$ as the \emph{interior vertices of $D$} and we refer to $w$ and $x$ as the \emph{exterior vertices of $D$}. We define the class $\mathcal{P}$ (the propellers) to be the class of graphs constructed as follows. Let $T_1, \dots, T_k$ be $k$ disjoint triangles and for each $i \in [k]$ label one vertex of triangle $T_i$ by $v_i$. Then $G \in \mathcal{P}$ if it is obtained from $T_1, \dots, T_k$ by identifying each vertex in $\{v_1, \dots, v_k\}$. Similarly, we define the class $\mathcal{W}$ as follows. Let $\mathcal{H}$ be the house graph and label the vertex on the triangle with degree $2$ $x$. Then $G \in \mathcal{W}$ if it is obtained from $T_1, \dots, T_k, \mathcal{H}$ by identifying each vertex in $\{v_1, \dots, v_k, x\}$. The identified vertex will be referred to as the ``nose" of the windmill/propeller.  Technically, $K_3$ is also a propeller with a nose and $\mathcal{H}$ is a windmill with a nose.

Next, the class $\mathcal{G}$ is built from the classes $\mathcal{P} \cup \mathcal{W}$ as follows. First, let $K^*$ be the graph obtained from the disjoint union of a diamond and $K_2$ by identifying a vertex of $K_2$ with an exterior vertex of the diamond. Then $\mathcal{G}$ contains $\mathcal{P}\cup \mathcal{W} \cup \{K^*\}$. All other graphs $G \in \mathcal{G}$ can be constructed in the following way. Let $G'=(A \cup B, E')$ be a connected, nontrivial, bipartite and well-edge-dominated graph with $|A| < |B|$.  Choose a set $B' \subset B$ and a set $Y \subseteq A$ such that the following are true:
\begin{itemize}
\item $B'\cup Y \ne \emptyset$.
\item  $G'- B'$ is a well-edge-dominated graph with $\gamma'(G'-B') = \gamma'(G')$, $G'-B'$ has no trivial components, and $|A| \le |B-B'|$. 
\item If $Y\ne \emptyset$, then we write $Y = \{y_1, \dots, y_{\ell}\}$ where $y_i$ is a support vertex in $G'- B'$. 
\end{itemize}
 Each vertex $v$ in $B'$ is called detachable and it is called strongly detachable if each neighbor of it in $G'$ is a support vertex in $G'-B'$.  If $Y\ne \emptyset$, let $D_1, \dots, D_{\ell}$ be a disjoint union of diamonds, and for each $D_i$, label one of the vertices of degree two as the ``nose" of the diamond. If $B' \ne\emptyset$, then choose a set $B'' \subseteq B'$ (possibly empty) where each vertex in $B''$ is strongly detachable. If $B'' \ne \emptyset$, then write $B'' = \{s_1, \dots, s_k\}$ and choose any set of $k$ windmills, enumerated as $W_1, \dots, W_k$. If $B' - B'' \ne \emptyset$, write $B' - B'' = \{x_1, \dots, x_r\}$ and choose any set of $r$ propellers, enumerated as $P_1, \dots, P_r$.  We obtain $G$ from $G'$ under the following rules: 
\begin{enumerate}
\item[(a)] The nose of $W_i$ is identified with $s_i$ in $B''$.
\item[(b)] The nose of $P_i$ is identified with $x_i$ in $B' - B''$. 
\item[(c)] The nose of $D_i$ is identified with the support vertex $y_i$ of $G'$ (which stays a support vertex in $G$). 
\end{enumerate}

Based on the above description, the following was shown.

\begin{theorem}\label{thm:K4free}\cite{AK-2026} Let $G$ be $K_4$-free. $G$ is well-edge-dominated with girth $3$ if and only if $G \in \{ \mathcal{DH}, F_5,Cr, W_1, W_2, W_3\}$ or $G \in \mathcal{G}$.
\end{theorem}

It turns out that the class $\mathcal{G}$ is the backbone of all graphs that are nonbipartite and well-edge-dominated. In the next section, we describe how to obtain a well-edge-dominated graph that contains an induced $K_4$. However, we first point out that while it was not explicitly stated in \cite{AK-2026}, we can calculate the edge domination number of $G \in \mathcal{G}$ based on how it was constructed. That is, suppose that $G$ was constructed from the connected, nontrivial, bipartite and well-edge-dominated graph $G'= (A\cup B, E')$ where $|A| < |B|$. Then a set $B' \subset B$ and $Y \subseteq A$ was chosen so that  \begin{itemize}
\item $B'\cup Y \ne \emptyset$.
\item  $G'- B'$ is a well-edge-dominated graph with $\gamma'(G'-B') = \gamma'(G')$, $G'-B'$ has no trivial components, and $|A| \le |B-B'|$. 
\item If $Y\ne \emptyset$, then we write $Y = \{y_1, \dots, y_{\ell}\}$ where $y_i$ is a support vertex in $G'- B'$. 
\end{itemize}
If $Y\ne \emptyset$, we enumerate $Y = \{y_1, \dots, y_{\ell}\}$ and let $D_1, \dots, D_{\ell}$ be the disjoint set of diamonds such that the nose of $D_i$ was identified with $y_i$. If $B'' \ne \emptyset$, we enumerate $B'' = \{s_1, \dots, s_k\}$ and let $W_1, \dots, W_k$ be the disjoint set of windmills such that the nose of $W_i$ was identified with $s_i$. If $B' - B'' \ne \emptyset$, we enumerate $B' - B'' = \{x_1, \dots, x_r\}$ and let $P_1, \dots, P_r$ be the set of disjoint propellers such that the nose of $P_i$ was identified with $x_i$. We can find a maximal matching in $G$ as follows. First, let $M$ be a maximal matching in $G'$ that saturates $A$ and contains all pendent edges in $G'$. If $Y\ne \emptyset$, let the interior vertices of $D_i$ be $u_i$ and $v_i$ so that $u_iv_i$ edge dominates all edges of $D_i$. If $B''\ne \emptyset$, we enumerate the vertices of the house graph of $W_i$ as $z_i, a_i, b_i, c_i$, and $d_i$ where $a_ib_ic_id_ia_i$ is a $4$-cycle and $z_i$ is the nose of the windmill, and all remaining triangles of $W_i$ are of the form $r_j^is_j^iz_i$ for $j \in [m_i]$. Finally, if $B'-B'' \ne \emptyset$, then we label the nose of $P_i$ as $z_i$ and all triangles of $P_i$ are of the form $g_j^ih_j^iz_i$ for $j \in [n_i]$. Thus, 
\[M \cup \{u_iv_i: i \in [\ell]\} \cup \{a_ib_i, c_id_i: i \in [k]\} \cup \left(\bigcup_{i=1}^k\{r_j^is_j^i: j \in [m_i]\}\right) \cup \left(\bigcup_{i=1}^r\{g_j^ih_j^i: j \in [n_i]\}\right)\] is a maximal matching in $G$. Therefore, 
\begin{equation}\label{eqn:1}
\gamma'(G) = \gamma'(G') + \overline{\ell} + \overline{k}+\overline{r}
\end{equation}
where
\begin{itemize}
\item $\overline{\ell} = \ell$ if $Y\ne \emptyset$; otherwise $\overline{\ell} =0$.
\item $\overline{k} = 2k + \sum_{i=1}^k(n(W_i)-5)/2 $ if $B''\ne \emptyset$; otherwise $\overline{k} =0$.
\item $\overline{r} = \sum_{i=1}^r(n(P_i)-1)/2$ if $B'-B''\ne \emptyset$; otherwise $\overline{r}=0$.
\end{itemize}

\section{Characterizing all nonbipartite well-edge-dominated graphs}

We are now at a point to finish characterizing all nonbipartite well-edge-dominated graphs by considering those with girth $3$ that also contain an induced $K_4$. To do so, we build upon the class $\mathcal{G}$ in the following way. Define the class $\mathcal{G^*}$ to be those graphs in $\mathcal{G}$ together with any graph $G$ which can be obtained in the following way. Let $\mathcal{B}$ represent the set of all connected, nontrivial, bipartite, and well-edge-dominated graphs such that if $J \in \mathcal{B}$, then we can write $J = (A_J\cup B_J, E_J)$ where $|A_J|<|B_J|$. Then $G\in \mathcal{G^*}$ if it is obtained from a graph $H \in \mathcal{G}\cup \mathcal{B}$ of order at least three by performing the following operation:
\vskip5mm
\noindent\textbf{The $K_4$-Operation:} Pick a pendent edge $uv \in E(H)$ where $u$ is the support vertex, $u$ is not on a triangle in $H$, and there is no minimal edge dominating set $F$ of $H$ containing an edge of the form $ut$ where $ut$ has a private edge neighbor of the form $ts$ for some $s\ne u$. Add vertices $w$ and $x$ as well as edges $wx, wu, wv, xu$, and $xv$.
\vskip5mm

All remaining graphs in $\mathcal{G^*}$ can be obtained from a graph $H \in \mathcal{G^*}$ by performing a $K_4$-Operation on $H$. Equivalently, we can view every graph in $\mathcal{G^*}$ as being obtained from a graph in $\mathcal{G} \cup \mathcal{B}$ by repeatedly performing the $K_4$-Operation.

In one direction, we show that if $G\in \mathcal{G^*}$, then $G$ is well-edge-dominated. To do so, we use the following observation that follows from Lemma \ref{lem:satA}.

\begin{obs} If $G = (A\cup B, E)$ is a well-edge-dominated bipartite graph where $|A| \leq |B|$, then any minimal edge dominating set $F$ satisfies $|F| = |A|$.
\end{obs}

\begin{lemma} For any $G \in \mathcal{G^*}$, $G$ is well-edge-dominated. 
\end{lemma}

\begin{proof} Note that from Theorem~\ref{thm:K4free}, we may assume that $\omega(G) = 4$. Let $H$ be an induced $K_4$ in $G$ with vertices $\{w, x, y, z\}$. We can view building $G$ by repeatedly performing  the $K_4$-Operation. So all we need to show is that when $G$ is obtained from $G' \in \mathcal{G^*}\cup \mathcal{B}$ by performing the $K_4$-Operation, then $G$ is also well-edge-dominated. Suppose  that $G' \in \mathcal{G^*}\cup \mathcal{B}$ and we obtain $G$ by performing the $K_4$-Operation on the  pendent edge $wy$  by adding vertices $x$ and $z$ and edges $wx, wz, xy, zy,$ and $xz$. Note that $\alpha'(G) = \alpha'(G')+1$ as any matching in $G'$ together with $xz$ is a matching in $G$. Therefore, $\Gamma'(G) \ge \alpha'(G')+1$. We need to show that any minimal edge dominating set $F$ of $G$ has cardinality $\alpha'(G')+1$, i.e. $\Gamma'(G) = \alpha'(G')+1$. Suppose to the contrary that $G$ contains a minimal edge dominating set $F$ with cardinality $\alpha'(G') + 2$. Note that $|F\cap \{xy, xw, xz, wz, wy, zy\}| \le 2$. Suppose first that $F$ contains an edge of the form $wb$ where $b \not\in \{x, y, z\}$. If $|F \cap \{xy, xw, xz, wz, wy, zy\}|=1$, then $F' = F - (F \cap \{xy, xw, xz, wz, wy, zy\})$ is a minimal edge dominating set of $G'$ meaning that $|F| - 1 = |F'| = \alpha'(G')$, which is a contradiction. On the other hand, if $|F \cap \{xy, xw, xz, wz, wy, zy\}|=2$, then we may assume that $F \cap \{xy, xw, xz, wz, wy, zy\}=\{wx, wy\}$. Moreover, $wb$ must have a private edge neighbor of the form $ba$. However, $F- \{wx, wy\}$ is a minimal edge dominating set of $G'$ where $wb$ has a private edge neighbor of the form $ba$, contradicting that we could perform the $K_4$-Operation on $wy$. Therefore, we shall assume that $F$ contains no such edge $wb$. It follows that  $|F \cap \{xy, xw, xz, wz, wy, zy\}| = 2$. In this case, $F' = (F - \{xy, xw, xz, wz, wy, zy\}) \cup \{wy\}$ is a minimal edge dominating set of $G'$. Hence, $|F|-1 = |F'| = \alpha'(G') $, yet another contradiction. It follows that any minimal edge dominating set of $G$ has cardinality $\alpha'(G') + 1$ and $G$ is well-edge-dominated.

\end{proof}

Next, we point out that there are three well-edge-dominated graphs, each with order $7$, that contain a $K_4$ and are not in the class $\mathcal{G^*}$. These graphs are named $J_1, J_2$, and $J_3$, depicted in Figure~\ref{fig:Js}. Moreover, $K_4$ is well-edge-dominated and is not in $\mathcal{G^*}$. This means that together with the six exceptions to the $K_4$-free well-edge-dominated graphs that both contain a triangle and are not in $\mathcal{G}$, we now have ten graphs that are well-edge-dominated, contain a triangle, and are not in $\mathcal{G^*}$. Four of them are $K_4, \mathcal{DH}, Cr$, and $F_5$. The remaining six are shown in Figure~\ref{fig:Js}. 

\begin{figure}[h!]
\begin{center}
\begin{tikzpicture}[scale=.75]
    \vertex (1) at (0,0)  [scale=.75, fill=black, label=below:$a$]{};
    \vertex (2) at (0,1)  [scale=.75, fill=black, label=left:$w$]{};
    \vertex (3) at (1, .5)  [scale=.75, fill=black,label=below:$y$]{};
    \vertex (4) at (1, 1.5)  [scale=.75, fill=black, label=above:$x$]{};
    \vertex (5) at (2, 1)  [scale=.75, fill=black, label=right:$z$]{};
    \vertex (6) at (2, 0)  [scale=.75, fill=black, label=below:$c$]{};
    \vertex (7) at (1, -.5)  [scale=.75, fill=black, label=below:$b$]{};
  
     \vertex (B1) at (4,0)  [scale=.75, fill=black, label=below:$a$]{};
    \vertex (B2) at (4,1)  [scale=.75, fill=black, label=left:$w$]{};
    \vertex (B3) at (5, .5)  [scale=.75, fill=black]{};
    \vertex (B4) at (5, 1.5)  [scale=.75, fill=black, label=above:$x$]{};
    \vertex (B5) at (6, 1)  [scale=.75, fill=black, label=right:$z$]{};
    \vertex (B6) at (6, 0)  [scale=.75, fill=black, label=below:$c$]{};
    \vertex (B7) at (5, -.5)  [scale=.75, fill=black, label=below:$b$]{};
    
       \vertex (C1) at (8,0)  [scale=.75, fill=black, label=below:$a$]{};
    \vertex (C2) at (8,1)  [scale=.75, fill=black, label=left:$w$]{};
    \vertex (C3) at (9, .5)  [scale=.75, fill=black,label=below:$y$]{};
    \vertex (C4) at (9, 1.5)  [scale=.75, fill=black, label=above:$x$]{};
    \vertex (C5) at (10, 1)  [scale=.75, fill=black, label=right:$z$]{};
    \vertex (C6) at (10, 0)  [scale=.75, fill=black, label=below:$c$]{};
    \vertex (C7) at (9, -.5)  [scale=.75, fill=black, label=below:$b$]{};

       \vertex (D1) at (12,0)  [scale=.75, fill=black, label=below:$a$]{};
    \vertex (D2) at (12,1)  [scale=.75, fill=black, label=left:$w$]{};
    \vertex (D3) at (13, .5)  [scale=.75, fill=black, label=below:$y$]{};
    \vertex (D4) at (13, 1.5)  [scale=.75, fill=black, label=above:$x$]{};
    \vertex (D5) at (14, 1)  [scale=.75, fill=black, label=right:$z$]{};
    \vertex (D6) at (14, 0)  [scale=.75, fill=black, label=below:$c$]{};
    \vertex (D7) at (13, -.5)  [scale=.75, fill=black, label=below:$b$]{};
    
       \vertex (1a) at (1.5,-5)  [scale=.75, fill=black, label=above:$y$]{};
    \vertex (2a) at (0,-6.5)  [scale=.75, fill=black, label=left:$x$]{};
    \vertex (3a) at (1.5,-8)  [scale=.75, fill=black, label=below:$w$]{};
    \vertex (4a) at (3,-6.5)  [scale=.75, fill=black, label=above:$z$]{};
    \vertex (5a) at (5,-6.5)  [scale=.75, fill=black, label=above:$b$]{};
    \vertex (6a) at (6.5,-5)  [scale=.75, fill=black, label=above:$c$]{};
    \vertex (7a) at (6.5,-8)  [scale=.75, fill=black, label=below:$a$]{};
    
       \vertex (1b) at (10,-8)  [fill, scale=.75, label=below:$x$]{};
    \vertex (2b) at (10,-6) [fill, scale=.75, label=left:$w$]{};
    \vertex (3b) at (13,-8)  [fill, scale=.75, label=below:$z$]{};
    \vertex (4b) at (13,-6)   [fill, scale=.75, label=right:$y$]{};
    \vertex (5b) at (10,-4) [fill, scale=.75, label=left:$a$]{};
    \vertex (6b) at (13,-4)  [fill, scale=.75, label=right:$c$]{};
    \vertex (7b) at (11.5, -3)  [fill, scale=.75, label=above:$b$]{};

    \node(A) at (1, -1.75) []{$W_1$};
    \node(B) at (5, -1.75)[]{$W_2$};
    \node(C) at (9, -1.75)[]{$W_3$};
     \node(D) at (13, -1.75)[]{$J_1$};
      \node(E) at (4, -9)[]{$J_2$};
      \node(F) at (11.5, -9)[]{$J_3$};
      \node(G) at (4.75,.25)[]{$y$};

    \path 
	(1) edge (2)
	(2) edge (3)
	(3) edge (4)
	(4) edge (5)
	(5) edge (6)
	(6) edge (7)
	(1) edge (7)
	(2) edge (4)
	(3) edge (5)
	
	(B1) edge (B2)
	(B2) edge (B3)
	(B3) edge (B4)
	(B4) edge (B5)
	(B5) edge (B6)
	(B6) edge (B7)
	(B1) edge (B7)
	(B2) edge (B4)
	(B3) edge (B5)
	(B3) edge (B7)
	
  	(C1) edge (C2)
	(C2) edge (C3)
	(C3) edge (C4)
	(C4) edge (C5)
	(C5) edge (C6)
	(C6) edge (C7)
	(C1) edge (C7)
	(C2) edge (C4)
	(C3) edge (C5)  
	(C4) edge (C6)
	
	(D1) edge (D2)
	(D2) edge (D3)
	(D3) edge (D4)
	(D4) edge (D5)
	(D2) edge (D4)
	(D3) edge (D5)  
	(D2) edge (D5)
	(D1) edge (D7)
	(D7) edge (D6)
	(D5) edge (D6)
	
	       (1a) edge (2a)
    (2a) edge (3a)
    (3a) edge (4a)
    (4a) edge (1a)
    (4a) edge (5a)
    (5a) edge (6a)
    (5a) edge (7a)
    (6a) edge (7a)
    (6a) edge (1a)
    (7a) edge (3a)
    (1a) edge (3a)
    (2a) edge (4a)
    
        (1b) edge (2b)
    (1b) edge (3b)
    (4b) edge (2b)
    (4b) edge (3b)
    (5b) edge (2b)
    (6b) edge (4b)
    (7b) edge (6b)
    (7b) edge (5b)
    (5b) edge (6b)
    (2b) edge (3b)
    (1b) edge (4b)

    ;

\end{tikzpicture}
\caption{The graphs $W_1$, $W_2$, $W_3$, $J_1$, $J_2$, and $J_3$}
\label{fig:Js}
\end{center}
\end{figure}
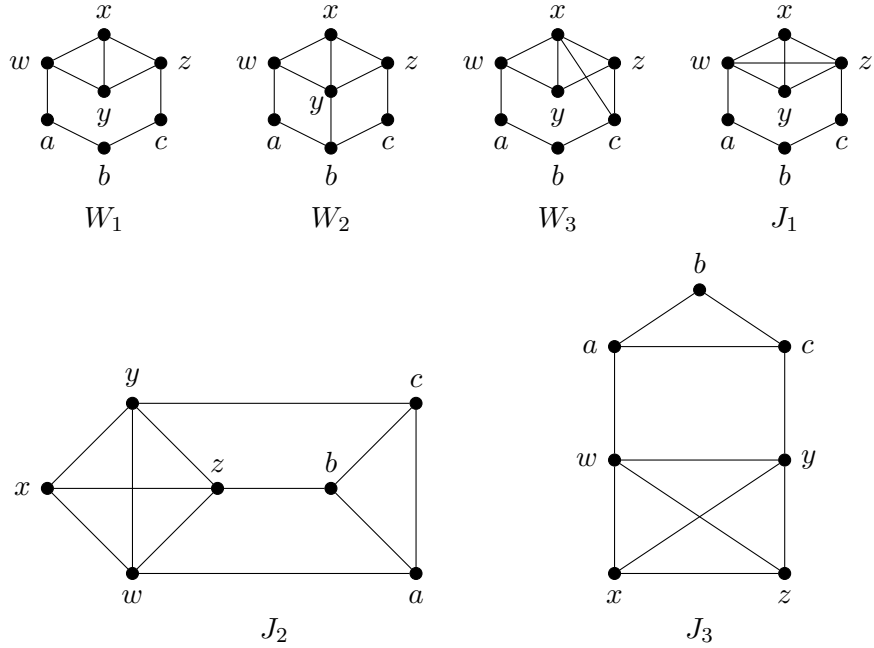

In order to show that we have found all connected well-edge-dominated graphs containing a triangle, we first prove a lemma about well-edge-dominated graphs with small diameter and clique number $4$.  

\begin{lemma}\label{lem:smalldiamond} Let $G$ be a graph containing an induced $K_4$ with vertices $\{w, x, y, z\}$. If at least two vertices in $\{w, x, y, z\}$ have degree at least $4$ in $G$ and every vertex in $G$ is either a leaf or adjacent to a vertex in $\{w, x, y, z\}$, then $G$ is not well-edge-dominated. 
\end{lemma}

\begin{proof} Suppose to the contrary that such a $G$ is well-edge-dominated and among such counterexamples, choose one of minimum order. We partition the vertices of $G$ as $V(G) = \{w, x, y, z\} \cup S(G) \cup L(G) \cup T(G)$ where the vertices in $S(G)$ are support vertices in $V(G) - \{w, x, y, z\}$, the vertices in $L(G)$ are leaves in $G$ which are not adjacent to some vertex in $\{w, x, y, z\}$, and $T(G)= V(G) - (S(G)\cup L(G) \cup \{w, x, y, z\})$. Let $M_S$ be a matching containing one pendent edge per vertex in $S(G)$ and let $M_T$ be a maximal matching in $G[T(G)]$. Thus, $M_S \cup M_T \cup \{xw, yz\}$ is a maximal matching in $G$ with cardinality $|M_S \cup M_T| +2$.

Assume that $S(G)\ne \emptyset$. First, suppose that there are distinct vertices $u, v \in S(G)$ such that $u$ is adjacent to some $a \in \{w, x, y, z\}$ and $v$ is adjacent to some $b \in \{w, x, y, z\} - \{a\}$. Then $\{ua, vb\} \cup M_T \cup \{cd\} \cup (M_S - \{e_1, e_2\})$ where $\{c, d\} = \{w, x, y, z\} - \{a, b\}$ and $e_1 \in M_S$ (resp. $e_2 \in M_S$) is the edge in $M_S$ incident to $u$ (resp. incident to $v$) is also a maximal matching in $G$ of cardinality $|M_S \cup M_T|+1$, which is a contradiction. Therefore, we may assume that every vertex in $S(G)$ is adjacent to exactly one vertex in $\{w, x, y, z\}$, say $x$. This implies that $T(G) \ne \emptyset$ for otherwise $\deg_G(z) =\deg_G(w) = \deg_G(y)= 3$. Also note that $G - N_e[M_S \cup \{xw, yz\}]= G[T(G)]$ is well-edge-dominated. 

Let $F = \{xs:s \in S(G)\}$ and let $H$ be the connected, nontrivial component in $G-N_e[F\cup M_T]$ which is isomorphic to the triangle $wyz$ together with all leaves in $T(G)$ as well as all vertices in $T(G)$ not incident to an edge in $M_T$. Suppose we can find three distinct vertices $v_1, v_2, v_3 \in V(H)$ where $v_1$ is adjacent to $w$, $v_2$ is adjacent to $y$ and $v_3$ is adjacent to $z$. Then $M_S \cup M_T \cup \{v_1w, v_2y, v_3z\}$ is a minimal edge dominating set of $G$ where with cardinality $|M_S \cup M_T| +3$, which is a contradiction. Therefore, we may assume $\deg_H(a) = 2$ for some $a \in \{w, y, z\}$. Without loss of generality, assume $\deg_H(w) = 2$. If $\deg_H(y) > 2$, $\deg_H(z)>2$, then we can find $v_1, v_2 \in V(H)$ where $F \cup M_T \cup \{yv_1, zv_2\}$ and $F \cup M_T \cup \{yz\}$ are two minimal edge dominating sets. So we shall assume $\deg_H(w) = \deg_H(y) = 2$. However, in this case, both $F \cup M_T \cup \{yz\}$ and $M_S \cup M_T \cup \{xw, yz\}$ are minimal edge dominating sets, another contradiction. Therefore, it must be that $S(G) = \emptyset$ and $M_T \cup \{xw, yz\}$ is a maximal matching in $G$. 

Enumerate $M_T = \{u_1v_1, u_2v_2, \dots, u_nv_n\}$ such that (if possible) $|(N_G(u_1)\cup N_G(v_1))\cap \{w, x, y, z\}| \ge2$. Assume first that $n\ge2$. If it is true that $|(N_G(u_1)\cup N_G(v_1))\cap \{w, x, y, z\}| \ge2$, then when we consider $H= G - N_e[u_2v_2]$, we know that $H$ is well-edge-dominated and the component $H'$ of $H$ containing $\{w, x, y, z\}$ has order $n(G) - 2$, every vertex in $V(H') - \{w, x, y, z\}$ is adjacent to at least one of $\{w, x, y, z\}$ and at least two vertices in $\{w, x, y, z\}$ have degree at least $4$. However, this contradicts the minimality of $G$, so this case cannot be. Therefore, $|(N_G(u_i)\cup N_G(v_i))\cap \{w, x, y, z\}| =1$ for each $i \in [n]$. With no loss of generality, assume $(N_G(u_1)\cup N_G(v_1))\cap \{w, x, y, z\} =\{z\}$. If there exists $2 \le j < k \le n$ such that $(N_G(u_j)\cup N_G(v_j))\cap \{w, x, y, z\} = \{a\}$ and $(N_G(u_k)\cup N_G(v_k))\cap \{w, x, y, z\} = \{b\}$ where $a\ne b$ and $\{a, b\} \cap \{z\} = \emptyset$, then again $G - N_e[u_1v_1]$ contains a component which is a smaller counterexample than $G$. Therefore, we may assume that for each $i\in [n]$ that $(N_G(u_i)\cup N_G(v_i))\cap \{w, x, y, z\}= \{z\}$ or $(N_G(u_i)\cup N_G(v_i))\cap \{w, x, y, z\}= \{x\}$. Moreover, it must be that $n=2$ for otherwise we reach a similar contradiction. Thus, we shall assume $(N_G(u_1)\cup N_G(v_1))\cap \{w, x, y, z\}= \{z\}$ and $(N_G(u_2)\cup N_G(v_2))\cap \{w, x, y, z\}= \{x\}$. Without loss of generality, we may assume $u_1$ is adjacent to $z$ and $u_2$ is adjacent to $x$. Note that $G- N_e[wx, yz]$ is well-edge-dominated with edge domination number two. Thus, every vertex in $T(G) - \{u_1, v_1, u_2, v_2\}$ is only adjacent to some subset of vertices in $\{w, x, y, z, u_1, u_2, v_1, v_2\}$. Moreover, as $G - N_e[u_2v_2]$ is well-edge-dominated and contains the $K_4$ induced by $\{w, x, y, z\}$, it follows that there exists exactly one vertex $a \in \{w, x, y, z\}$ such that every vertex in $T(G) - \{u_2, v_2\}$ is adjacent to only $a$ in $\{w, x, y, z\}$ for otherwise we have a smaller counterexample than $G$. Since $u_1$ is adjacent to $z$, then $a=z$. On the other hand, the same is true for $G - N_e[u_1v_1]$, which would imply that $a=x$ as $u_2$ is adjacent to $x$. This contradiction shows that we may assume $T(G) = \{u_1, v_1, u_2, v_2\}$. In this case, $G[T(G)]$ is either $2K_2$, $K_4$, or $C_4$. If $G[T(G)] \cong C_4$, then assume it is $u_1v_1u_2v_2u_1$. In either case, $\{xu_2, wy, zu_1\}$ and $\{wx, yz, u_1v_1, u_2v_2\}$ are  maximal matchings contradicting that $G$ is equimatchable. Therefore, this case cannot occur and it must be that $|M_T|=1$.

Finally, assume that $M_T = \{uv\}$. Thus, there exists a maximal matching in $G$ of size $3$. First note that $G - N_e[wx, yz]$ is well-edge-dominated with edge domination one, meaning that it is the union of isolates and either a star $K_{1, \ell}$ where we may assume that $u$ is the center of the star, or $K_3$. Note also that we may assume that $n(G) \ge 8$ so that either $T(G)$ contains leaves or $\ell >1$. As above, since $G - N_e[uv]$ is well-edge-dominated, all vertices in $T(G) - \{u,v\}$ are adjacent to the same vertex in $\{w, x, y, z\}$ for otherwise the component containing $\{w, x, y, z\}$ in $G - N_e[uv]$ is a smaller counterexample than $G$. With no loss of generality, let us assume that for each  $t \in T(G) - \{u,v\}$ that $N_G(t) \cap \{w, x, y, z\} = \{z\}$. By assumption, it means that either $u$ or $v$ is adjacent to a vertex in $\{w, x, y\}$. Assume first that $ux \in E(G)$. Enumerate the vertices of $T(G) -\{u, v\} = \{b_1, \dots, b_{\ell-1}\}$. If the nontrivial component in $G[T(G)]$ is a star and $\ell \ge 3$, then $\{xw, xy, b_1u, b_2u, \dots, b_{\ell-1}u\}$ is a minimal edge dominating set of cardinality strictly larger than $3$ if $z$ is not adjacent to $u$ or $v$ and $\{xw, xy, uv, b_1u, b_2u, \dots, b_{\ell-1}u\}$ is a minimal edge dominating set of cardinality strictly larger than $3$ if $z$ is adjacent to $u$ or $v$.  Therefore, we may assume that the nontrivial component in $G[T(G)]$ is either $P_2, K_3$, or $K_{1, 2}$. As $n(G) \ge 8$, this implies that there is some $r \in T(G)$ which is a leaf of $z$. If $u$ has a neighbor in $T(G)$ other than $v$, then we call it $t$. If $t$ exists and $v$ is adjacent to a vertex in $\{x, y, w\}$, then $G - N_e[tu]$ is a counterexample to the statement of the theorem, which is a contradiction. On the other hand, if $t$ exists and $N(v) \cap \{w, x, y, z\} = \{z\}$, then one of  $\{zv, zw, zx, zt\}$, $\{zv, zy, zx, zt\}$, or $\{zv, zw, zx, zt, zy\}$  is a minimal edge dominating set (depending on $N(u)\cap \{w, x, y, z\}$). However, this implies $G$ is not well-edge-dominated as $\gamma'(G)=3$. Thus, we may assume that $t$ does not exist. Note that $G-N_e[wz]$ contains a component with vertices $\{x, y, u, v\}$, meaning that it must induce a $4$-cycle or $K_4$. Therefore, we may assume that $vy\in E(G)$. Note that $\{xz, vy\}$ is a maximal matching in $G$ unless $uw\in E(G)$. Therefore, we may assume $uw\in E(G)$. In this case, $\{uw, yz\}$ is a maximal matching in $G$ unless $vx \in E(G)$. Therefore, we may assume $uw, vx \in E(G)$. However, now $\{rz, uv, vx, vy\}$ is a minimal edge dominating set in $G$, which is a contradiction. In every case, we reach a contradiction. Therefore, we may assume that $ux \not\in E(G)$. Moreover, a similar argument can be used to show $uy\not\in E(G)$ and $uw\not\in E(G)$. Thus, $N_G(u) \cap \{w, x, y, z\} = \{z\}$. It follows that we may assume $vx \in E(G)$. If $u$ contains a neighbor in $T(G)$ other than $v$, call it $t$ again, then $G- N_e[tu]$ contains a component that is a smaller counterexample. So we may assume that $uv$ is the only edge in $G[T(G)]$. But this is identical to the case that we just considered above with merely the roles of $u$ and $v$ interchanged. Hence, no such counterexample exists.

\end{proof}

We are now ready to show that we have a complete characterization of all nonbipartite well-edge-dominated graphs. Figures \ref{ND7} and \ref{ND8} show all connected well-edge dominated graphs containing a triangle of order $7$ and $8$. There are two additional graphs of order $7$ that are well-edge dominated, which are $C_7$ and $C_7^*$, and there are no additional graphs of order $8$ that are well-edge dominated. We can also make the following observations from the figures.

\begin{obs} If $G$ is well-edge-dominated with order at most $8$, then $G$ does not contain an induced $C_5$ unless  $G \in \{C_5, J_1, W_3, Cr, C_7^*\}$. 
\label{case1}
\end{obs}

\begin{obs} If $G$ is well-edge-dominated with order at most $8$, then $G$ does not contain an induced $\mathcal{H}$ unless $G \in \{W_2, W_7, W_8, W_{10}, W_{12}, W_{15}, W_{16},V_7\}$. 
\label{case2}
\end{obs}

\begin{theorem} $G$ is a nonbipartite well-edge-dominated graph if and only if $G \in \mathcal{G^*}$ or $G \in \{ K_4, C_5, C_7, C_7^*,  \mathcal{DH}, F_5, Cr,  W_1, W_2, W_3, J_1, J_2, J_3\}$.
\end{theorem}

\begin{proof} Suppose to the contrary that the statement of the theorem is false and let $G$ be a minimum counterexample with respect to order. Using SAGE, we know if $n(G) \leq 8$, then $G$ satisfies the statement of the theorem. In addition, by Theorems \ref{thm:girth5red} and \ref{thm:K4free}, we know if $G$ is $K_4$-free, then $G$ satisfies the statement of the theorem. Thus, we may assume that $n(G) \ge 9$ and that $G$ contains an induced $K_4$.  Let $D$ be the set of vertices $t\in V(G)$ such that if we choose an induced $K_4$ with vertices $\{w, x, y, z\}$ and  we contract $w, x, y$, and $z$ into a single vertex $T$, then $\dist(T, t)$ is maximum among all choices of induced $K_4$s  and vertices $t$. Then for each $t \in D$, and among all neighbors of $t$ in $G$, choose $s$ so that $\dist(T,s)$ is maximum. We choose $t$ and $s$ such that $t \in D$ and $s \in N_G(t)$ such that $\dist(T, s)$  is maximum among all $t\in D$. Consider $G' = G - N_e[st]$. We may write $G' = G_0 \cup I$ where $I$ is an independent set in $G$ and $G_0$ is connected for otherwise $st$ was chosen incorrectly. We can partition $I = I_s \cup I_{st} \cup I_t$ where each vertex in $I_s$ is adjacent to $s$ and not $t$, each vertex in $I_t$ is adjacent to $t$ and not $s$, and each vertex in $I_{st}$ is adjacent to both $s$ and $t$. We also claim we can make the following assumptions.
\begin{enumerate}
\item We may assume $I_t \cup I_{st} = \emptyset$. Suppose to the contrary that $\ell \in I_t \cup I_{st}$. If $\dist(T, s) < \dist(T, t)$, then $\dist(T, \ell) \ge \dist(T, t)$ which would imply that we should have chosen $s = \ell$. On the other hand, if $\dist(T, s) = \dist(T, t)$, then $\dist(T, \ell) > \dist(T, t)$ and we incorrectly chose $t$. In either case, we reach a contradiction. So we may assume that $I_t \cup I_{st} = \emptyset$.
\item  We may assume that $|I_s| \le 1$. If $|I_s|\ge 2$, then for some $\ell \in I_s$, $G - \ell$ is also well-edge-dominated, contradicting our choice of minimal counterexample. \item We may assume that $t$ is not adjacent to any vertex in $\{w, x, y, z\}$. This assumption follows from  Lemma~\ref{lem:smalldiamond}.
\item We may assume if $s$ is adjacent to some vertex in $\{w, x, y, z\}$, then $t$ is not a leaf. Suppose for the sake of contraction that $t$ is a leaf. By Lemma~\ref{lem:smalldiamond}, if every vertex in $G$ has a maximum distance of $2$ from some vertex in $\{w, x, y, z\}$, then there exists some edge $st$ where $t$ is not adjacent to any vertex in $\{w, x, y, z\}$ and $t$ is not a leaf.
\end{enumerate}
By our choice of $G$ and $G_0$, $G_0$ is well-edge dominated with order at least $6$, contains an induced $K_4$, and does satisfy the statement of the theorem. Thus, we proceed by considering whether $G_0 \in \{J_1, J_2, J_3\}$ or $G_0 \in \mathcal{G^*}$. 
\vskip2mm
\noindent\textbf{Case 1:} Suppose that $G_0 \cong J_1$.
\vskip2mm
Label the vertices as in Figure~\ref{fig:Js}. Note that $G - N_e[xy]$ contains the induced $C_5$ with vertices $\{w, z, a, b, c\}$. Thus, the component $\widetilde{G}$ of $G-N_e[xy]$ containing $\{w, z, a, b, c\}$ must be isomorphic to $C_5$,  $J_1$, $W_3$, $Cr$, or  $C_7^*$ by Observation \ref{case1}. However, if $\widetilde{G}\cong C_5$, then neither $s$ nor $t$ is adjacent to any vertex in $\{w, z, a, b, c\}$, meaning that $s$ is adjacent to either $x$ or $y$ and $t$ is a leaf in $G$, which is a contradiction. Next, we shall assume that $\widetilde{G} \cong J_1$. As $t$ is not adjacent to $w$ or $z$, we may assume  with no loss of generality, that $\{a, b, s, t\}$ induces a $K_4$ in $G$. In this case, $\{xz, sb, aw\}$ is a maximal matching in $G$ as well as $\{xz, wy, ab, st\}$, which cannot be. Therefore, this case cannot occur.

Next, assume $\widetilde{G}\cong W_3$. Based on how $s$ and $t$ are defined, we may assume without loss of generality that $wabczw$ represents the $5$-cycle, and either (i) $s$ is adjacent to each vertex in $\{w, a\}$ and $t$ is adjacent to each vertex in $\{a,b\}$, or (ii) $s$ is adjacent to each vertex in $\{a,b\}$ and $t$ is adjacent to each vertex in $\{b,c\}$. If  (i) is true, then $\{sw, yz, bt\}$ and $\{sw, yz, bc, at\}$ are both minimal edge dominating sets of $G$ (regardless of whether $s$ is adjacent to $x$ or $y$), which is a contradiction. If (ii) is true, then $\{bt, wx, as, cz\}$ and $\{st, wz, wx, ct, aw\}$ are both minimal edge dominating sets of $G$ (regardless of whether $s$ is adjacent to $x$ or $y$), which is a contradiction.



Next, assume $\widetilde{G}\cong Cr$. Then $\widetilde{G}$ has two induced subgraphs isomorphic to $C_5$.  It is clear from $G_0$ that the induced $5$-cycle $wabczw$ in $\widetilde{G}$  has exactly two vertices on the triangle and the remaining three vertices not on the triangle. This produces two cases: (i) if $s$ is the third vertex of the triangle, then $wzs$ is the triangle in $Cr$ (as $t$ is not adjacent to $w$ or $z$) and $t$ is adjacent to $a$ and $c$, or (ii) if $t$ is the third vertex of the triangle, then without loss of generality, we may assume $tab$ is the triangle in $Cr$ and $s$ is adjacent to $w$ and $c$. In case (i), $\{xz, at, ab, ws\}$ and $\{ab, at, sz, xz, yz\}$ are both minimal edge dominating sets of $G$ (regardless of whether $s$ is adjacent to $x$ or $y$), which is a contradiction. In case (ii), $\{st, wx, ab, cz\}$ and $\{sw, zy, tb\}$ are both minimal edge dominating sets of $G$ (regardless of whether $s$ is adjacent to $x$ or $y$), which is a contradiction.

Finally, assume $\widetilde{G}\cong C_7^*$ where $wabczw$ is the induced $5$-cycle in $C_7^*$. Based on the definition of $s$ and $t$, we may assume without loss of generality that either (i) $s$ is adjacent to $w$ and $t$ is adjacent to $a$,  or (ii) $s$ is adjacent to $a$ and $t$ is adjacent to $b$. If $s$ is adjacent to $w$ and $t$ is adjacent to $a$, then $\{st, cz, xy, aw\}$ and $\{st, at, cz, yz, xz\}$ are both minimal edge dominating sets of $G$ (regardless of whether $s$ is adjacent to $x$ or $y$), which is a contradiction. If $s$ is adjacent to $a$ and $t$ is adjacent to $b$, then $\{st, ab, wx, yz\}$ and $\{st, as, cz, wz, yz\}$ are both minimal edge dominating sets of $G$ (regardless of whether $s$ is adjacent to $x$ or $y$), which is a contradiction.

\vskip2mm
\noindent\textbf{Case 2:} Suppose that $G_0 \cong J_2 $ or $G_0 \cong J_3$.
\vskip2mm 
Label the vertices as in Figure~\ref{fig:Js} and note that the only difference between $J_2$ and $J_3$ is the edge $zb$ in $J_2$. Consider $H=G-N_e[xz]$ which contains the induced subgraph isomorphic to the house graph with vertices  $\{w, y, a, b, c\}$ by Observation \ref{case2}. Suppose first that $H$ is connected. From Figures 3 and 4, it must be that $H \in \{W_2, W_7, W_8, W_{10}, W_{12}, W_{15}, W_{16},V_7\}$. Recall that $t$ cannot be adjacent to $w, y, x,$ or $z$; however, $s$ may be. We may assume that $H \not\in\{W_{10}, W_{12}, W_{16}\}$ as this would imply that $t$ is adjacent to either $w$ or $y$ since $abc$ is a triangle in $H$. 
\begin{itemize}
\item If $H \cong W_2$, then we shall assume that $t$ is adjacent to $a$ and $b$, and $s$ is adjacent to $w$. Then both $\{st, yz, yx, bc, ta\}$ and $\{st, yz, aw, ab\}$ are minimal edge dominating sets of $G$ (regardless of whether $s$ is adjacent to $x$ or $z$), which is a contradiction.

\item If $H \cong W_7$, then we shall assume that $sb \in E(G)$ and $t$ is the leaf in $H$. This implies that $t$ is a leaf in $G$ as $t$ is not adjacent to $x$ or $z$. However, both $\{wa, yz, sb\}$ and $\{xz, wy, ac, st\}$ are maximal matchings, which is a contradiction.

\item If $H \cong W_8$, then we shall assume that both $s$ and $t$ are adjacent to $b$. Since $t$ is not adjacent to $x$ or $z$,  both $\{wz, cy, st, ab\}$ and $\{wz, cy,sb\}$ are maximal matchings of $G$ (regardless of whether $s$ is adjacent to $x$ or $z$), which is a contradiction.

\item Assume $H \cong W_{15}$. Note that if the triangle $abc$ in $H$ is not contained in the $K_4$ in $W_{15}$, then $t$ would necessarily be adjacent to $w$ in $H$, which is a contradiction to our choice of $s$ and $t$. Therefore, we may assume that  $\{a, b, c, t\}$ induces the $K_4$ in $H$ and $s$ is adjacent to each of $\{y, w\}$. Then $\{zy, st, ab, sw\}$ and $\{ab, xz, wx, ct, st\}$ are minimal edge dominating sets of $G$ (regardless of whether $s$ is adjacent to $x$ or $z$), which is a contradiction. 

\item If $H \cong V_7$, then we shall assume that  $s$ is adjacent to $b$ and both $t$ and $\ell$ are leaves in $H$. It follows that $t$ and $\ell$ are leaves in $G$ as $t$ is not adjacent to $x$ or $z$. Then $\{xw, cy, st, ab\}$ and $\{xw, sb, cy\}$ are both maximal matchings of $G$ (regardless of whether $s$ is adjacent to $x$ or $z$), which is a contradiction.

\end{itemize}
Having exhausted all cases, we may assume that $H$ is disconnected, meaning that $s$ is adjacent to a vertex in $\{x, z\}$. Suppose first that $G_0 \cong J_2$. It follows that $G$ satisfies the statement of Lemma~\ref{lem:smalldiamond} as $\{w, x, y, z\}$ induces a $K_4$ with at least two vertices of degree $4$ in $G$, every vertex other than $t$ and $\ell$ (should it exist) has a neighbor in $\{w, x, y, z\}$, and both $t$ and $\ell$ (should it exist) are leaves, meaning $G$ is not well-edge-dominated. Therefore, this case cannot occur either and $G_0 \not\cong J_2$. On the other hand, if $G_0 \cong J_3$, we may assume that $xs \in E(G)$ as $J_3$ is symmetric and $t$ is a leaf in $G$. Then $\{ac, wy, xz, st\}$ is a maximal matching in $G$ and yet $\{ac, wy, xs\}$ is a minimal edge dominating set in $G$, which is a contradiction. Thus, this case cannot occur and $G_0 \not\cong J_3$. 
\vskip2mm

\vskip2mm
\noindent\textbf{Case 3:} Finally, we may assume that $G_0 \in \mathcal{G^*}$.
\vskip2mm 

We shall assume that $G_0$ is obtained from a graph in $ \mathcal{G^*}$ by the $K_4$-Operation involving $\{w, x, y, z\}$. We shall also assume that $\deg_{G_0}(x) = \deg_{G_0}(y) = \deg_{G_0}(z) = 3$ and $\deg_{G_0}(w) \ge 4$, and $G_0$ was built from constructing a $K_4$ from the edge $wz$. Recall that $t$ is not adjacent to any vertex in $\{x, y, z, w\}$. We first claim that we may assume that $s$ is not adjacent to any vertex in $\{x, y, z\}$. Suppose to the contrary that $s$ is adjacent to $a \in \{x, y, z\}$. In this case, $t$ is not a leaf in $G$. In addition, every vertex in $V(G) - \{w, x, y, z\}$ is either adjacent to a vertex in $\{w, x, y, z\}$ or is distance two from a vertex in $\{w, x, y, z\}$ by our choice of $s$ and $t$. Assume $G_0$ contains another triangle induced by $\{p, q, r\}$ in $V(G_0) - \{w, x, y, z\}$. By construction of  $\mathcal{G^*}$, we can assume neither $q$ nor $r$ is adjacent to a vertex in $\{w, x, y, z\}$. However, this implies we should have chosen the edge $qr$, contradicting our choice of $st$, which is a contradiction. Therefore, we may assume that $G_0$ contains no such triangle $pqr$.   This implies that $G_0$ is built from a bipartite graph $J$ containing the pendent edge $wz$ and adding the vertices $x$ and $y$ and edges $xw, xy, xz, yw,$ and $yz$. Thus, we can write $J = (A_J\cup B_J, E)$ where $|A_J|<|B_J|$ and $z \in B_J$ and $w \in A_J$. Moreover, every vertex in $G_0$ is assumed to be within distance two of $w$, meaning that every vertex in $B_J$ is dominated by $w$. If there exists some $a \in A_J - \{w\}$, then  we can choose a matching $M$ of size $|A_J|-1$ from $G_0$ that saturates $A_J - \{w\}$ implying that $G_0 - N_e[M]$ is a graph satisfying the statement of Lemma \ref{lem:smalldiamond}, which is a contradiction. It follows that $J$ is a star. Moreover, $G_0$ has order at least 7 meaning that $G_0$ contains at least 3 leaves. Since $t$ is not a leaf in $G$ and $t$ is not adjacent to $\{x, y, z, w\}$, $t$ is adjacent to some $r \in B_J- \{z\}$, implying $G - N_e[rt]$ satisfies the statement of Lemma~\ref{lem:smalldiamond}, another contradiction. Hence, we may assume that $s$ is not adjacent to any vertex in $\{x, y, z\}$. That is, we may assume that $\deg_G(x) = \deg_G(y) = \deg_G(z) =3$.

 Consider $\widetilde{G} = G - N_e[yz]$, which is necessarily connected and contains the leaf $x$. Since $\widetilde{G}$ is a connected, well-edge dominated graph with order at least  $6$, $\widetilde{G} \in \{C_7, C_7^*,  \mathcal{DH}, Cr,  W_1, W_2, W_3, J_1, J_2, J_3\}$ or $\widetilde{G} \in \mathcal{G^*}$. Since $x$ is a leaf, $\widetilde{G} \in \mathcal{G^*}$. Then $G$ is obtained from $\widetilde{G}$ by performing the $K_4$-Operation provided that (a) $w$ is not on a triangle in $\widetilde{G}$, and (b) there is no minimal edge dominating set $F$ of $\widetilde{G}$ containing an edge of the form $wb$ where $wb$ has a private edge neighbor of the form $ab$ where $a \ne w$. Note that $G_0 \in \mathcal{G^*}$ by the assumption in this case. If $w$ is on a triangle in $\widetilde{G}$, but $w$ is not on a triangle in $G_0$, then $s$ (and $w$) is on said triangle. Thus, $w$ and $s$ must have degree at least $3$ in $\widetilde{G}$. Since $\widetilde{G} \in \mathcal{G^*}$, the triangle containing $s$ and $w$ must be on a diamond or a $K_4$ based on the construction of $\mathcal{G^*}$. Furthermore, by our choice of $st$ and the fact that $s$ is adjacent to $w$, it must be that every vertex in $\widetilde{G}$ has a maximum distance of two from $w$ in both $\widetilde{G}$ and $G$. Therefore, we may assume $\widetilde{G}$ is built from a well-edge-dominated bipartite graph $J = (A_J\cup B_J, E)$ where $|A_J|<|B_J|$ and $x$ is a leaf in $J$ so $w$ is a support vertex in $J$. However, since $J = (A_J\cup B_J, E_J)$ is well-edge-dominated with $|A_J| < |B_J|$, by Lemma 1 no vertex in $A_J$ is a leaf in $J$ meaning that every vertex in $B_J$ is adjacent to $w$. Suppose first that there is some $a \in A_J - \{w\}$ and let $ab \in E_J$. Then $G- N_e[ab]$ is well-edge-dominated and yet it contains the $K_4$ induced by $\{w, x, y, z\}$ as well as a diamond or $K_4$ induced by the vertices $\{s, t, w, r\}$ for some $r \in V(G)$, which does not match the description of any graph in $\mathcal{G^*}$. Therefore, we may assume that $J$ is a star. If  $\widetilde{G}$ contains a diamond, then  $\{sr, wx, yz\}$ and $\{ws, wr, wx, yz\}$ are minimal edge dominating sets of $G$, which is a contradiction. On the other hand, if $\widetilde{G}$ does not contain a diamond, then $\{ws, xy\}$ and $\{st, wx, yz\}$ are both minimal edge dominating set of $G$, another contradiction. Therefore, (a) is true. 
 
To see that (b) is true, suppose there is some minimal edge dominating set $F$ of $\widetilde{G}$ containing the edge $wb$ where $wb$ has the private edge neighbor $ab$ for some $a \in V(G) - \{w, x, y, z, b\}$. In this case, both $F \cup \{xy\}$ and $F \cup \{wx, wy\}$ are minimal edge dominating sets of $G$, which contradicts the assumption that $G$ is well-edge-dominated. Therefore, (b) is true as well and we may conclude that $G \in \mathcal{G^*}$, the last contradiction. Thus, no such counterexample exists. 
 
\end{proof}
\section{Line graphs of nonbipartite well-edge-dominated graphs}
In this section, we focus on the line graph for each $G\in \mathcal{G^*}$. There is a natural relationship between well-edge-dominated graphs and well-dominated graphs stated in the following observation.

\begin{obs}\label{obs:wed} For any well-edge-dominated graph $G$, $\mathcal{L}(G)$ is well-dominated.
\end{obs}

Recall that Vizing's conjecture states that for any pair of graphs $G$ and $H$, $\gamma(G\Box H) \ge \gamma(G)\gamma(H)$. One of the first results on this conjecture was proven by Barcalkin and German \cite{BG-1979} whereby they showed that if $G$ was such that $\gamma(G) = \theta(G)$ (i.e. if $G$ has a clique partition containing $\gamma(G)$ cliques) then $G$ satisfies Vizing's conjecture. We refer to any graph having the property that $\gamma(G) = \theta(G)$ as being a ``BG graph". From Observation~\ref{obs:wed}, we know that if $G$ is well-edge-dominated then $\gamma(\mathcal{L}(G)) = \alpha(\mathcal{L}(G))$. Thus, if $G$ is bipartite and well-edge-dominated, then Observation~\ref{obs:bipartite} yields $\mathcal{L}(G)$ is perfect and hence $\alpha(\mathcal{L}(G)) = \theta(\mathcal{L}(G))$, meaning $\gamma(\mathcal{L}(G)) = \theta(\mathcal{L}(G))$. Hence, we have the following observation. 

\begin{obs}\label{lem:bipartiteBG} If $G$ is a bipartite and well-edge-dominated graph, then $\mathcal{L}(G)$ is a BG graph.
\end{obs}


Before we move on, we point out that there are  several classes of graphs which are relevant to Vizing's conjecture. If $v \in V(G)$ satisfies that $\gamma(G-v) = \gamma(G)-1$, then we refer to $v$ as being a \emph{critical vertex with respect to domination}. We say $G$ is \emph{vertex critical} (with respect to domination) if every vertex is critical. Similarly, if for any $uv \not\in E(G)$ it is true that $\gamma(G+\{uv\}) = \gamma(G)-1$, then we refer to $G$ as being \emph{edge critical with respect to domination}. Bre\v{s}ar et al. \cite{VCsurvey} point out that if there exists a minimal (with respect to order) counterexample to Vizing's conjecture (meaning that  $\gamma(G\Box H)<\gamma(G)\gamma(H)$ for some graph $H$) that we may assume that $G$ is edge critical. Furthermore, if $\gamma(G+\{uv\})<\gamma(G)$ for every $uv\not\in E(G)$, then either $u$ or $v$ is critical in $G$. Thus, vertex criticality and edge criticality are not unrelated, yet behave very differently. For more about these two topics, see \cite{ACMM-1996, BH-2011, RJ-2011, SB-1981}.

Now if we take a well-edge-dominated graph \[G\in \mathcal{G^*}\cup\{K_4, C_5, C_7, C_7^*, Cr, F_5, \mathcal{DH}, J_1, J_2, J_3, W_1, W_2, W_3\},\] one can verify that there exist edges that cannot be removed from $G$ to maintain a well-edge-dominated graph. For instance, take any edge on a propeller that is not incident to the nose of the propeller. This implies that we can find $e \in E(G)$ such that $\gamma'(G-e) <\gamma'(G)$. And since we know its line graph $\mathcal{L}(G)$ is well-dominated, this implies that the corresponding vertex $v_e \in V(\mathcal{L}(G))$ satisfies $\gamma(G-v_e) < \gamma(G)$. That is, $\mathcal{L}(G)$ will contain critical vertices. In fact, there are some graphs $G\in \mathcal{G^*}\cup\{K_4, C_5, C_7, C_7^*, Cr, F_5, \mathcal{DH}, J_1, J_2, J_3, W_1, W_2, W_3\}$ such that every vertex in $\mathcal{L}(G)$ is critical. For instance, one can verify that for any edge $e \in E(W_1)$ that $\gamma'(G-e) < \gamma'(G)$, meaning that $\mathcal{L}(W_1)$ is a vertex critical graph with respect to domination. 

We now focus on all graphs in $\mathcal{G^*}$ and show that the line graph of each graph in this class satisfies Vizing's conjecture. In what follows, we will be going back and forth between a graph and its line graph. If we have a graph $G$ and an edge $e \in E(G)$, we will use $v(e)$ to refer to the vertex in $\mathcal{L}(G)$ corresponding to $e$. 

\begin{theorem} If $G \in \mathcal{G^*}$, then $\mathcal{L}(G)$ is a spanning subgraph of a BG graph $\widetilde{G}$ where $\gamma(\mathcal{L}(G)) = \gamma(\widetilde{G})$.
\end{theorem}

\begin{proof} We proceed by considering whether or not $G$ contains a $K_4$. 
\vskip5mm
\noindent\textbf{Claim 1:} If $G$ is $K_4$-free, then $\mathcal{L}(G)$ is a spanning subgraph of a BG graph.
\vskip5mm

\noindent\textit{Proof.}  Since $G \in \mathcal{G^*}$ and $G$ is $K_4$-free, we may assume that $G$ was constructed from the connected, nontrivial, bipartite and well-edge-dominated graph $G'= (A\cup B, E')$ where $|A| < |B|$. Then a set $B' \subset B$ and $Y \subseteq A$ was chosen so that  \begin{itemize}
\item $B'\cup Y \ne \emptyset$.
\item  $G'- B'$ is a well-edge-dominated graph with $\gamma'(G'-B') = \gamma'(G')$, $G'-B'$ has no trivial components, and $|A| \le |B-B'|$. 
\item If $Y\ne \emptyset$, then we write $Y = \{y_1, \dots, y_{\ell}\}$ where $y_i$ is a support vertex in $G'- B'$. 
\end{itemize}
If $Y\ne \emptyset$, we enumerate $Y = \{y_1, \dots, y_{\ell}\}$ and let $D_1, \dots, D_{\ell}$ be the disjoint set of diamonds such that the nose of $D_i$ was identified with $y_i$. If $B'' \ne \emptyset$, we enumerate $B'' = \{s_1, \dots, s_k\}$ and let $W_1, \dots, W_k$ be the disjoint set of windmills such that the nose of $W_i$ was identified with $s_i$. If $B' - B'' \ne \emptyset$, we enumerate $B' - B'' = \{x_1, \dots, x_r\}$ and let $P_1, \dots, P_r$ be the set of disjoint propellers such that the nose of $P_i$ was identified with $x_i$. We can find a maximal matching in $G$ as follows. First, let $M$ be a maximal matching in $G'$ that saturates $A$ and contains all pendent edges in $G'$. If $Y\ne \emptyset$, let the interior vertices of $D_i$ be $u_i$ and $v_i$ so that $u_iv_i$ edge dominates all edges of $D_i$, and let the other exterior vertex of $D_i$ be $w_i$. If $B''\ne \emptyset$, we enumerate the vertices of the house graph of $W_i$ as $z_i, a_i, b_i, c_i$, and $d_i$ where $a_ib_ic_id_ia_i$ is a $4$-cycle, $z_ia_ib_i$ is a triangle, $z_i$ is the nose of the windmill, and all remaining triangles of $W_i$ are of the form $r_j^it_j^iz_i$ for $j \in [m_i]$. Finally, if $B'-B'' \ne \emptyset$, then we label the nose of $P_i$ as $z_i$ and all triangles of $P_i$ are of the form $g_j^ih_j^iz_i$ for $j \in [n_i]$. As mentioned in Equation~\ref{eqn:1}, we know that $$\gamma'(G) = |A| +\overline{ \ell} + \overline{k} + \overline{r},$$
where $\overline{\ell} = \ell$ if $Y\ne \emptyset$ and $\overline{\ell} = 0$ otherwise; $\overline{k} = 2k+ \sum_{i=1}^k(n(W_i)-5)/2$ if $B''\ne\emptyset$ and $\overline{k} =0$ otherwise; $\overline{r} = \sum_{i=1}^{r}(n(P_i)-1)/2$ if $B'-B''\ne\emptyset$ and $\overline{r}=0$ otherwise. We proceed based on the cardinality of $B''$. 
\vskip2mm
\noindent\textbf{Case 1:} Suppose first that $B''=\emptyset$. 
\vskip2mm

We must find a clique partition of $\mathcal{L}(G)$ of cardinality $\gamma(\mathcal{L}(G))$. Note in this case that  $\overline{k} =0$. Regardless of whether $Y\ne \emptyset$, we can enumerate $A = \{y_1, \dots, y_{\ell}, y_{\ell+1}, \dots, y_{|A|}\}$ such that if $G$ contains an induced diamond, then $u_i$ and $v_i$ of $D_i$ are adjacent to $y_i$. We can also enumerate $B- B' = \{b_1, \dots, b_{\ell}, b_{\ell+1}, \dots, b_n\}$ such that if $Y\ne \emptyset$, then $b_i$ is a leaf of $y_i$ that stays a leaf in $G$. As above, we know by Observation~\ref{lem:bipartiteBG} that $\mathcal{L}(G')$ is a BG graph and has a clique partition $V(\mathcal{L}(G')) = \theta_1 \cup \dots \cup \theta_{|A|}$ where each $\theta_i$ induces a clique in $\mathcal{L}(G')$. 

Suppose for the time being that $Y\ne \emptyset$. We may assume that if $1 \le i<i'\le \ell$ that $v(y_ib_i)$ and $v(y_{i'}b_{i'})$ are in different cliques in $\theta$. Therefore we may reindex if necessary so that for each $i\in[\ell]$ $v(y_ib_i) \in \theta_i$. Note further that $v(u_iy_i)$ and $v(v_iy_i)$ are adjacent to each vertex in $\theta_i$ as any vertex in $\theta_i$ must represent an edge in $G$ that is incident to $y_i$ (since $y_ib_i$ is a pendent edge). We define $\theta'_i = \theta_i \cup \{v(u_iy_i), v(v_iy_i)\}$. Moreover, $G[v(u_iw_i), v(v_iw_i), v(u_iv_i)]$ is a triangle in $\mathcal{L}(G)$. If $B' - B'' = \emptyset$, then  $\overline{r}=0$ and \[\theta' = \theta_1' \cup \cdots \cup \theta_{\ell}' \cup \theta_{\ell+1} \cup \cdots \cup \theta_{|A|} \cup \{v(u_iv_i), v(u_iw_i), v(v_iw_i): i \in [\ell]\}\]
is a clique partition of $\mathcal{L}(G)$ of cardinality $|A| + \ell$. On the other hand, if $B' - B''\ne \emptyset$, then $\theta'$ which is the union of 
 \[\theta_1' \cup \cdots \cup \theta_{\ell}' \cup \theta_{\ell+1} \cup \cdots \cup \theta_{|A|}\cup \{v(u_iv_i), v(u_iw_i), v(v_iw_i): i \in [\ell]\}  \]
 together with 
 \[ \bigcup_{i=1}^r \{v(g_j^ih_j^i), v(g_j^iz_i), v(h_j^iz_i): j \in [n_i]\}\]
is a clique partition of cardinality $|A| + \ell + \sum_{i=1}^r (n(P_i) -1)/2= |A| + \overline{\ell} + \overline{r}$. In either case, $\mathcal{L}(G)$ is a BG graph.

Finally, note that if $Y = \emptyset$, then $\overline{\ell} = 0$. Since in this case $B' = Y = \emptyset$ and by construction $B'' \cup Y \neq \emptyset$, it must be that $B'-B'' \ne \emptyset$. In this case,
 \[\theta' = \theta_1\cup \cdots  \cup \theta_{|A|} \cup  \bigcup_{i=1}^r \{v(g_j^ih_j^i), v(g_j^iz_i), v(h_j^iz_i): j \in [n_i]\}\]
is a clique partition of cardinality $|A| + \sum_{i=1}^r (n(P_i) -1)/2= |A| + \overline{r}$. 
\vskip2mm
\noindent\textbf{Case 2:} Suppose that $B''\ne \emptyset$.
\vskip2mm

If $B'-B''=\emptyset$, then we can write $B' = B'' = \{s_1, \dots, s_k\}$ where $s_i$ is identified with the nose of $W_i$ for each $i\in [k]$ when building $G$. Otherwise, $B' = \{x_1, \dots, x_r, s_1, \dots, s_k\}$ where $s_i$ is identified with the nose of $W_i$ for each $i\in [k]$ and $x_i$ is identified with the nose of $P_i$ for each $i\in [r]$. First, we define a supergraph $\widetilde{G}$ of $\mathcal{L}(G)$. For each $i\in [k]$, we focus on the line graph of $W_i$. We obtain $\widetilde{G}$ from $\mathcal{L}(G)$ by adding the edge $v(b_iz_i)v(c_id_i)$ for each $i \in [k]$. We claim that $\gamma(\widetilde{G}) = \gamma(\mathcal{L}(G))$. Suppose to the contrary that $\gamma(\widetilde{G}) < \gamma(\mathcal{L}(G))$. Thus, there is a dominating set $D$ of $\widetilde{G}$ containing either $v(b_iz_i)$ or $v(c_id_i)$ for some $i \in [k]$ where $|D| < \gamma(\mathcal{L}(G))$. Suppose first that $v(b_iz_i) \in D$. It follows that some $x \in \{v(a_iz_i), v(a_ib_i), v(a_id_i), v(c_id_i)\}$ is in $D$ in order to dominate $v(a_id_i)$. Moreover, there is at least one vertex in $D$ of the form $v(r_j^it_j^i), v(t_j^iz_i)$, or $v(r_j^iz_i)$ for each $j \in [m_i]$. If we let $F$ be the set of edges in $G$ corresponding to the vertices in $D$, then $F - (E(G)\cap E(W_i))$ edge dominates $H$ where $H$ is the component of $G-s_i$ that  contains any vertex in $V(G')-\{s_i\}$ and has cardinality at most $\gamma'(G) - 3 - m_i$. If $H$ is bipartite, then we know that $H$ is well-edge-dominated from the condition that $G'- s_i$ is well-edge-dominated. On the other hand, if $H$ is not bipartite, then $H \in \mathcal{G}$ as it is built from $G'-s_i$ (which is well-edge-dominated) by choosing $\widetilde{B'} \subset B - \{s_i\}$ and $Y\subseteq A$ such that 
\begin{itemize}
\item $\widetilde{B'}\cup Y \ne \emptyset$
\item $(G'-s_i) - \widetilde{B'}$ is well-edge-dominated with $\gamma'((G'-s_i)-\widetilde{B'}) = \gamma'(G'-s_i) = \gamma'(G')$ and $(G'-s_i)-\widetilde{B'}$ has no trivial components where $|A| \le |B-\widetilde{B'}|$
\end{itemize}
It follows that  $\gamma'(H) = \gamma'(G) - 2 - m_i$. Therefore, there cannot be an edge dominating set of $H$ of cardinality $\gamma'(G') - 3 - m_i$. It follows that $\gamma(\widetilde{G}) = \gamma(\mathcal{L}(G))$. Similarly, if we had instead assumed that $v(c_id_i) \in D$, then some vertex $x \in \{v(b_ic_i), v(a_id_i), v(a_ib_i), v(z_ib_i), v(z_ia_i)\}$ is in $D$ in order to dominate $v(a_ib_i)$, and again we arrive at the same contradiction as before. 

All we need to do now is find a clique partition $\theta'$ of $\widetilde{G}$ of cardinality $\gamma(\widetilde{G}) = |A| +\overline{ \ell} + \overline{k} + \overline{r}$. Suppose first that $Y\ne \emptyset$. As above, we let $\theta= \theta_1\cup \cdots \cup\theta_{|A|}$ be a clique partition of $\mathcal{L}(G')$ where we may assume that for $i\in[\ell]$ we have $v(y_ib_i) \in \theta_i$. For each $i\in [\ell]$, we define  $\theta'_i = \theta_i \cup \{v(u_iy_i), v(v_iy_i)\}$. If $B' - B'' = \emptyset$, then $\overline{r} = 0$ and
\[\theta' = \theta_1' \cup \cdots \cup \theta_{\ell}' \cup \theta_{\ell+1} \cup \cdots \cup \theta_{|A|} \cup \{v(u_iv_i), v(u_iw_i), v(v_iw_i): i \in [\ell]\}\cup \bigcup_{i=1}^k \Omega_i\]
where $\Omega_i = \{v(b_iz_i), v(b_ic_i), v(c_id_i)\} \cup \{v(a_ib_i), v(a_iz_i), v(a_id_i)\}$ if $W_i$ is the house graph and otherwise \[\Omega_i = \{v(b_iz_i), v(b_ic_i), v(c_id_i)\} \cup \{v(a_ib_i), v(a_iz_i), v(a_id_i)\} \cup\{v(r_j^it_j^i), v(r_j^iz_i)v(t_j^iz_i): j\in [m_i]\}\] is a clique partition of $\widetilde{G}$ with cardinality $|A| +\overline{ \ell} + \overline{k}$. 

If $B' - B'' \ne \emptyset$, then 
 \[\theta' = \theta_1' \cup \cdots \cup \theta_{\ell}' \cup \theta_{\ell+1} \cup \cdots \cup \theta_{|A|} \cup \{v(u_iv_i), v(u_iw_i), v(v_iw_i): i \in [\ell]\} \cup \bigcup_{i=1}^r \Pi_i \cup\bigcup_{i=1}^k \Omega_i\]
where $\Omega_i$ is defined as above and $\Pi_i =  \{v(g_j^ih_j^i), v(g_j^iz_i), v(h_j^iz_i): j \in [n_i]\}$ is a clique partition of $\widetilde{G}$ with cardinality $|A| +\overline{ \ell} + \overline{k} + \overline{r}$. 

Finally, suppose that $Y=\emptyset$. In this case,  $\overline{\ell} = 0$. Let $\theta= \theta_1\cup \cdots \cup\theta_{|A|}$ be a clique partition of $\mathcal{L}(G')$. If $B' - B''= \emptyset$, then $\overline{r} = 0$ and
\[\theta' = \theta_1\cup \cdots \cup \theta_{|A|} \cup\bigcup_{i=1}^k \Omega_i\]

where $\Omega_i = \{v(b_iz_i), v(b_ic_i), v(c_id_i)\} \cup \{v(a_ib_i), v(a_iz_i), v(a_id_i)\}$ if $W_i$ is the house graph and otherwise \[\Omega_i = \{v(b_iz_i), v(b_ic_i), v(c_id_i)\} \cup \{v(a_ib_i), v(a_iz_i), v(a_id_i)\} \cup\{v(r_j^it_j^i), v(r_j^iz_i)v(t_j^iz_i): j\in [m_i]\}\] is a clique partition of $\widetilde{G}$ with cardinality $|A| + \overline{k}$. On the other hand, if $B' - B'' \ne \emptyset$, then 
\[\theta' = \theta_1\cup \cdots \cup \theta_{|A|} \cup  \bigcup_{i=1}^r \Pi_i \cup\bigcup_{i=1}^k \Omega_i\]
where $\Omega_i$ is defined as above and $\Pi_i =  \{v(g_j^ih_j^i), v(g_j^iz_i), v(h_j^iz_i): j \in [n_i]\}$ is a clique partition of $\widetilde{G}$ with cardinality $|A| +\overline{k} + \overline{r}$.
\smallqed

\vskip5mm
For the remainder of the proof, we may assume that $G$ contains a $K_4$. That is, we may assume that $G\in \mathcal{G^*}$ is obtained from $G' \in \mathcal{G}\cup \mathcal{B}$ by performing a sequence of the $K_4$-Operation. This means that we can find a set of pendent edges $\{a_1b_1, \dots, a_jb_j\}$ in $G'$ such that we have performed the $K_4$-Operation on each $a_ib_i$ for $i \in [j]$ and $b_i$ is a leaf in $G'$ for $i\in [j]$. We shall assume that for $a_ib_i$ that we have added the additional vertices $x_i$ and $y_i$ and all the edges of the form $\{a_ix_i, a_iy_i, x_iy_i, b_iy_i, b_ix_i\}$. 

Suppose first that $G' \in \mathcal{B}$. That is, assume that  $G' = (A\cup B, E)$ is a connected, nontrivial, well-edge-dominated bipartite graph where $|A|<|B|$. This means that we can write $A = \{a_1, \dots, a_j, a_{j+1}, \dots, a_m\}$ and $B = \{b_1, \dots, b_j, b_{j+1}, \dots, b_n\}$ where $m<n$. Note that since $G$ is well-edge-dominated, $\gamma'(G) = \alpha'(G)$. Moreover, we can find a maximal matching for $G$ by taking any maximal matching of $G'$ that saturates $A$ together with $\{x_iy_i: i \in [j]\}$. Thus, $\gamma'(G) = |A| + j$. It follows that $\gamma(\mathcal{L}(G)) = |A|+j$. All we need to show is that we can find a clique partition of $\mathcal{L}(G)$ of cardinality $|A|+j$. First, note that by Observation~\ref{lem:bipartiteBG} we know that $\mathcal{L}(G')$ is a BG graph and has a clique partition $V(\mathcal{L}(G')) = \theta_1 \cup \dots \cup \theta_{|A|}$ where each $\theta_i$ induces a clique in $\mathcal{L}(G')$. Note that for each $i\in [j]$ the vertex in $\mathcal{L}(G')$ associated with $a_ib_i$ is in some clique $\theta_{\ell}$ for some $\ell \in [|A|]$. Additionally, if $1\le i < i'\le j$, then $v(a_ib_i)$ and $v(a_{i'}b_{i'})$ are in different cliques of the clique partition. So we may reindex if necessary so that $v(a_ib_i)$ is contained in $\theta_i$ for $i \in [j]$. Now consider the edges in the $i^{th}$ $K_4$ in $G$. As $v(a_ib_i) \in \theta_i$, then $v(a_ix_i)$ and $v(a_iy_i)$ are each adjacent to all vertices in $\theta_i$ since $a_ib_i$ is a pendent edge. Therefore, we can define $\theta'_i= \theta_i \cup \{v(a_ix_i), v(a_iy_i)\}$. Moreover, $G[\{v(x_iy_i), v(x_ib_i),v(y_ib_i)\}]$ is a triangle in $\mathcal{L}(G)$. It follows that 
\[\theta' = \theta_1' \cup \dots \cup \theta_j' \cup \theta_{j+1}\cup \cdots \cup \theta_{|A|} \cup \bigcup_{i=1}^j \{v(x_ib_i),v(y_ib_i),v(x_iy_i)\}\]
is a clique partition of $\mathcal{L}(G)$ of cardinality $|A| + j$.

Finally, assume that $G' \in \mathcal{G}$. From Claim 1, we know that $\mathcal{L}(G')$ is a spanning subgraph of a BG graph $\widetilde{G}$ with $\gamma(\mathcal{L}(G')) = \gamma(\widetilde{G})$ and there is some partition $V(\widetilde{G}) = \theta_1 \cup \cdots \cup \theta_{|M|}$ where each $\theta_i$ induces a clique in $\widetilde{G}$, $|M| = \gamma'(G')$, and $v(a_ib_i) \in \theta_i$ for $i\in[j]$. As in Claim 1, if $B''= \emptyset$, then we set $H = \mathcal{L}(G)$, and if $B''\ne \emptyset$, then we define $H$ to be the graph obtained from $\mathcal{L}(G)$ by adding the edge $v(b_iz_i)v(c_id_i)$ for each $i\in [k]$. Using the same argument as above, one can show that indeed $\gamma(H) = \gamma(\mathcal{L}(G))$. Furthermore, if for $i\in [j]$ we set $\theta_i' = \theta_i \cup \{v(x_ia_i), v(y_ia_i)\}$, then 
\[\theta'_1 \cup \cdots \cup \theta_j' \cup \theta_{j+1} \cup \cdots \cup \theta_{|M|} \cup \bigcup_{i\in[j]}\{v(x_ib_i), v(y_ib_i), v(x_iy_i)\}\]
is a clique partition of $H$. It follows that $G$ is a spanning subgraph of the BG graph $H$.

\end{proof}

\begin{corollary} If $G \in \mathcal{G^*}$, then $\mathcal{L}(G)$ satisfies Vizing's conjecture. In particular, for any graph $H$, $\gamma(\mathcal{L}(G)\Box H) \ge \gamma(\mathcal{L}(G))\gamma(H)$.
\end{corollary}

\section{Concluding Remarks}

Based on the results on Section 3, we have a structural description of the nonbipartite, well-edge-dominated graphs. One may think that finding a structural description for the bipartite, well-edge-dominated graphs is trivial. However, we have not been able to make any progress on that front. This is particularly frustrating considering bipartite, well-edge-dominated graphs form the backbone of the nonbipartite, well-edge-dominated graphs. Therefore, we pose this as a problem. 

\vskip5mm
\noindent\textbf{Problem 1:} \emph{Provide a structural description of the class of all bipartite, well-edge-dominated graphs.}
\vskip5mm

Next, we point out that the set of all well-edge-dominated graphs is a proper subset of the set of all equimatchable graphs. Now that we can describe any nonbipartite, well-edge-dominated graph, is it possible to find a structural description of all nonbipartite, equimatchable graphs which are not well-edge-dominated? We pose this as another problem. 

\vskip5mm
\noindent\textbf{Problem 2:} \emph{Provide a structural description of the class of all nonbipartite, equimatchable graphs that are not well-edge-dominated. }
\vskip5mm

Throughout the paper, we also considered the vertex versions of these problems by considering well-dominated graphs. Classifying well-dominated graphs of girth $4$ has proven to be quite difficult. However, we were able to create an infinite class of well-dominated graphs with girth $3$ through the line graphs of well-edge-dominated graphs. Is it possible to characterize well-dominated graphs with girth $4$ that are not induced subgraphs of a well-dominated line graph with the same domination number? That is, perhaps one should focus on maximal well-dominated graphs $G$ of girth $4$ in that for any $uv \not\in E(G)$, $G+\{uv\}$ is not well-dominated. In this way, one would focus on edge criticality in well-dominated graphs. We pose this as another problem.

\vskip5mm
\noindent\textbf{Problem 3:} \emph{Provide a structural description of the class of all maximal well-dominated graphs of girth $4$.)}

\newpage

\subsection{Connected, well-edge-dominated graphs with $7$ vertices}

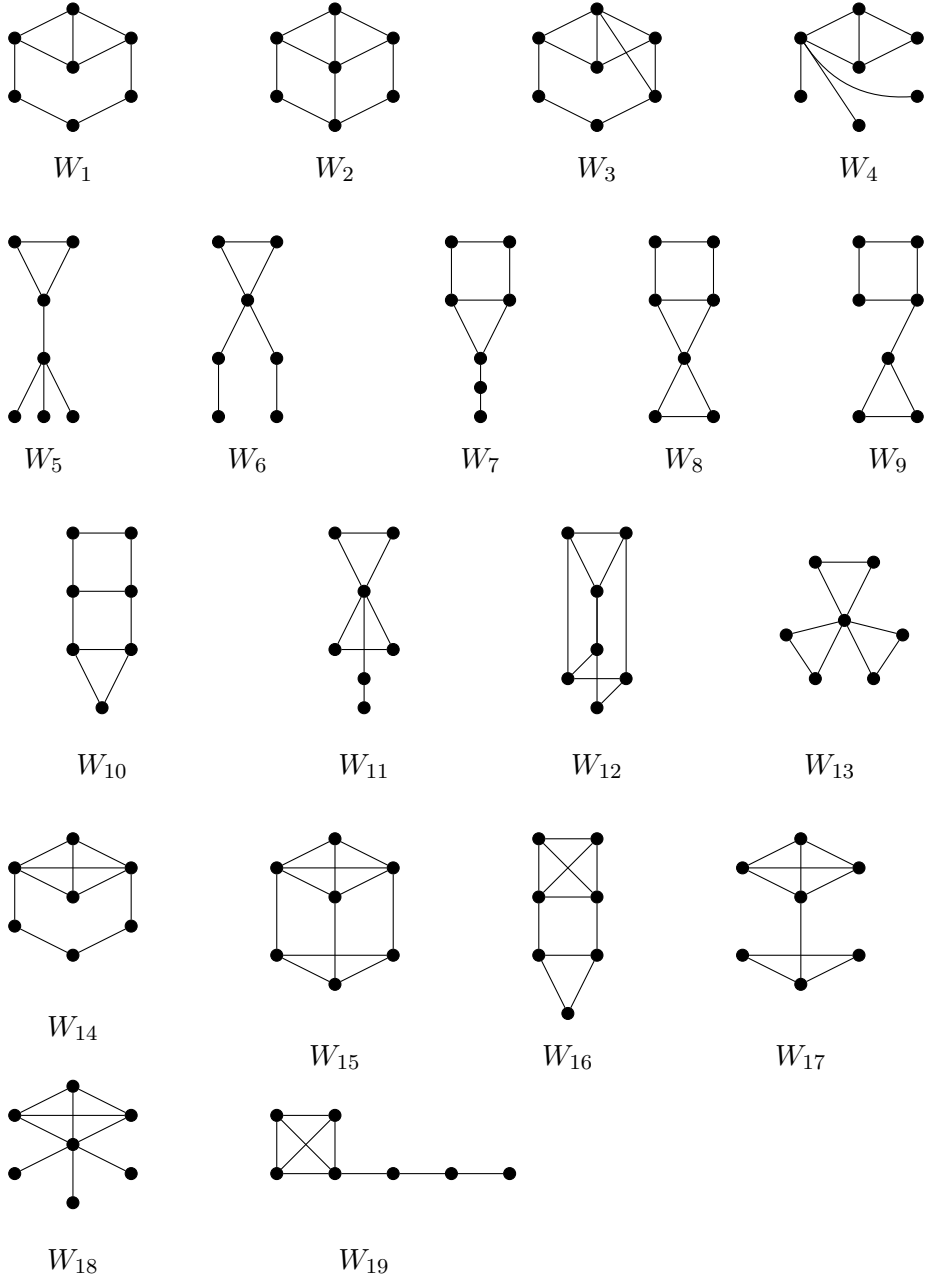
\begin{figure}[h!]
\begin{center}
\begin{tikzpicture}[scale=.77]

    \vertex (01) at (0.5,7)  [scale=.75, fill=black, label=below:$$]{};
    \vertex (02) at (0.5,8)  [scale=.75, fill=black, label=left:$$]{};
    \vertex (03) at (1.5, 7.5)  [scale=.75, fill=black]{};
    \vertex (04) at (1.5, 8.5)  [scale=.75, fill=black, label=above:$$]{};
    \vertex (05) at (2.5, 8)  [scale=.75, fill=black, label=right:$$]{};
    \vertex (06) at (2.5, 7)  [scale=.75, fill=black, label=below:$$]{};
    \vertex (07) at (1.5, 6.5)  [scale=.75, fill=black, label=below:$$]{};
  
     \vertex (0B1) at (5,7)  [scale=.75, fill=black, label=below:$$]{};
    \vertex (0B2) at (5,8)  [scale=.75, fill=black, label=left:$$]{};
    \vertex (0B3) at (6, 7.5)  [scale=.75, fill=black]{};
    \vertex (0B4) at (6, 8.5)  [scale=.75, fill=black, label=above:$$]{};
    \vertex (0B5) at (7, 8)  [scale=.75, fill=black, label=right:$$]{};
    \vertex (0B6) at (7, 7)  [scale=.75, fill=black, label=below:$$]{};
    \vertex (0B7) at (6, 6.5)  [scale=.75, fill=black, label=below:$$]{};
    
       \vertex (0C1) at (9.5,7)  [scale=.75, fill=black, label=below:$$]{};
    \vertex (0C2) at (9.5,8)  [scale=.75, fill=black, label=left:$$]{};
    \vertex (0C3) at (10.5, 7.5)  [scale=.75, fill=black]{};
    \vertex (0C4) at (10.5, 8.5)  [scale=.75, fill=black, label=above:$$]{};
    \vertex (0C5) at (11.5, 8)  [scale=.75, fill=black, label=right:$$]{};
    \vertex (0C6) at (11.5, 7)  [scale=.75, fill=black, label=below:$$]{};
    \vertex (0C7) at (10.5, 6.5)  [scale=.75, fill=black, label=below:$$]{};

       \vertex (0D1) at (14,7)  [scale=.75, fill=black, label=below:$$]{};
    \vertex (0D2) at (14,8)  [scale=.75, fill=black, label=left:$$]{};
    \vertex (0D3) at (15, 7.5)  [scale=.75, fill=black]{};
    \vertex (0D4) at (15, 8.5)  [scale=.75, fill=black, label=above:$$]{};
    \vertex (0D5) at (16, 8)  [scale=.75, fill=black, label=right:$$]{};
    \vertex (0D6) at (16, 7)  [scale=.75, fill=black, label=below:$$]{};
    \vertex (0D7) at (15, 6.5)  [scale=.75, fill=black, label=below:$$]{};

    \node(A) at (1.5, 5.75) []{$W_1$};
    \node(B) at (6, 5.75)[]{$W_2$};
    \node(C) at (10.5, 5.75)[]{$W_3$};
    \node(D) at (15, 5.75)[]{$W_4$};

    \vertex (1) at (.5,4.5)  [scale=.75, fill=black, label=below:]{};
    \vertex (2) at (1.5, 4.5)  [scale=.75, fill=black, label=left:]{};
    \vertex (3) at (1, 3.5)  [scale=.75, fill=black]{};
    \vertex (4) at (1, 2.5)  [scale=.75, fill=black, label=above:]{};
    \vertex (5) at (.5, 1.5)  [scale=.75, fill=black, label=right:]{};
    \vertex (6) at (1, 1.5)  [scale=.75, fill=black, label=below:]{};
    \vertex (7) at (1.5, 1.5)  [scale=.75, fill=black, label=below:]{};
  
     \vertex (B1) at (4,4.5)  [scale=.75, fill=black, label=below:]{};
    \vertex (B2) at (5,4.5)  [scale=.75, fill=black, label=left:]{};
    \vertex (B3) at (4.5, 3.5)  [scale=.75, fill=black]{};
    \vertex (B4) at (4, 2.5)  [scale=.75, fill=black, label=above:]{};
    \vertex (B5) at (5, 2.5)  [scale=.75, fill=black, label=right:]{};
    \vertex (B6) at (4, 1.5)  [scale=.75, fill=black, label=below:]{};
    \vertex (B7) at (5, 1.5)  [scale=.75, fill=black, label=below:]{};

        \vertex (C1) at (8,4.5)  [scale=.75, fill=black, label=below:]{};
    \vertex (C2) at (9,4.5)  [scale=.75, fill=black, label=left:]{};
    \vertex (C3) at (8, 3.5)  [scale=.75, fill=black]{};
    \vertex (C4) at (9, 3.5)  [scale=.75, fill=black, label=above:]{};
    \vertex (C5) at (8.5, 2.5)  [scale=.75, fill=black, label=right:]{};
    \vertex (C6) at (8.5, 2)  [scale=.75, fill=black, label=below:]{};
    \vertex (C7) at (8.5, 1.5)  [scale=.75, fill=black, label=below:]{};

        \vertex (D1) at (11.5,4.5)  [scale=.75, fill=black, label=below:]{};
    \vertex (D2) at (12.5,4.5)  [scale=.75, fill=black, label=left:]{};
    \vertex (D3) at (11.5, 3.5)  [scale=.75, fill=black]{};
    \vertex (D4) at (12.5, 3.5)  [scale=.75, fill=black, label=above:]{};
    \vertex (D5) at (12, 2.5)  [scale=.75, fill=black, label=right:]{};
    \vertex (D6) at (11.5, 1.5)  [scale=.75, fill=black, label=below:]{};
    \vertex (D7) at (12.5, 1.5)  [scale=.75, fill=black, label=below:]{};
    
     \vertex (E1) at (15,4.5)  [scale=.75, fill=black, label=below:]{};
    \vertex (E2) at (16,4.5)  [scale=.75, fill=black, label=left:]{};
    \vertex (E3) at (15, 3.5)  [scale=.75, fill=black]{};
    \vertex (E4) at (16, 3.5)  [scale=.75, fill=black, label=above:]{};
    \vertex (E5) at (15.5, 2.5)  [scale=.75, fill=black, label=right:]{};
    \vertex (E6) at (15, 1.5)  [scale=.75, fill=black, label=below:]{};
    \vertex (E7) at (16, 1.5)  [scale=.75, fill=black, label=below:]{};
    
         \vertex (F1) at (1.5,-.5)  [scale=.75, fill=black, label=below:]{};
    \vertex (F2) at (2.5,-.5)  [scale=.75, fill=black, label=left:]{};
    \vertex (F3) at (1.5, -1.5)  [scale=.75, fill=black]{};
    \vertex (F4) at (2.5, -1.5)  [scale=.75, fill=black, label=above:]{};
    \vertex (F5) at (1.5, -2.5)  [scale=.75, fill=black, label=right:]{};
    \vertex (F6) at (2.5, -2.5)  [scale=.75, fill=black, label=below:]{};
    \vertex (F7) at (2, -3.5)  [scale=.75, fill=black, label=below:]{};

         \vertex (G1) at (6,-.5)  [scale=.75, fill=black, label=below:]{};
    \vertex (G2) at (7,-.5)  [scale=.75, fill=black, label=left:]{};
    \vertex (G3) at (6.5, -1.5)  [scale=.75, fill=black]{};
    \vertex (G4) at (6, -2.5)  [scale=.75, fill=black, label=above:]{};
    \vertex (G5) at (7, -2.5)  [scale=.75, fill=black, label=right:]{};
    \vertex (G6) at (6.5, -3)  [scale=.75, fill=black, label=below:]{};
    \vertex (G7) at (6.5, -3.5)  [scale=.75, fill=black, label=below:]{};

         \vertex (H1) at (10,-.5)  [scale=.75, fill=black, label=below:]{};
    \vertex (H2) at (11,-.5)  [scale=.75, fill=black, label=left:]{};
    \vertex (H3) at (10.5, -1.5)  [scale=.75, fill=black]{};
    \vertex (H4) at (10.5, -2.5)  [scale=.75, fill=black, label=above:]{};
    \vertex (H5) at (10, -3)  [scale=.75, fill=black, label=right:]{};
    \vertex (H6) at (11, -3)  [scale=.75, fill=black, label=below:]{};
    \vertex (H7) at (10.5, -3.5)  [scale=.75, fill=black, label=below:]{};
    
      \vertex (I1) at (14.25,-1)  [scale=.75, fill=black, label=below:]{};
    \vertex (I2) at (15.25,-1)  [scale=.75, fill=black, label=left:]{};
    \vertex (I3) at (14.75, -2)  [scale=.75, fill=black]{};
    \vertex (I4) at (13.75, -2.25)  [scale=.75, fill=black, label=above:]{};
    \vertex (I5) at (14.25, -3)  [scale=.75, fill=black, label=right:]{};
    \vertex (I6) at (15.25, -3)  [scale=.75, fill=black, label=below:]{};
    \vertex (I7) at (15.75, -2.25)  [scale=.75, fill=black, label=below:]{};
    
        \vertex (J1) at (0.5,-7.25)  [scale=.75, fill=black, label=below:$$]{};
    \vertex (J2) at (0.5,-6.25)  [scale=.75, fill=black, label=left:$$]{};
    \vertex (J3) at (1.5, -6.75)  [scale=.75, fill=black]{};
    \vertex (J4) at (1.5, -5.75)  [scale=.75, fill=black, label=above:$$]{};
    \vertex (J5) at (2.5, -6.25)  [scale=.75, fill=black, label=right:$$]{};
    \vertex (J6) at (2.5, -7.25)  [scale=.75, fill=black, label=below:$$]{};
    \vertex (J7) at (1.5, -7.75)  [scale=.75, fill=black, label=below:$$]{};
    
            \vertex (K1) at (5,-6.25)  [scale=.75, fill=black, label=below:$$]{};
    \vertex (K2) at (6,-5.75)  [scale=.75, fill=black, label=left:$$]{};
    \vertex (K3) at (7, -6.25)  [scale=.75, fill=black]{};
    \vertex (K4) at (6, -6.75)  [scale=.75, fill=black, label=above:$$]{};
    \vertex (K5) at (5, -7.75)  [scale=.75, fill=black, label=right:$$]{};
    \vertex (K6) at (7, -7.75)  [scale=.75, fill=black, label=below:$$]{};
    \vertex (K7) at (6, -8.25)  [scale=.75, fill=black, label=below:$$]{};
    
                \vertex (L1) at (9.5,-5.75)  [scale=.75, fill=black, label=below:$$]{};
    \vertex (L2) at (9.5,-6.75)  [scale=.75, fill=black, label=left:$$]{};
    \vertex (L3) at (9.5, -7.75)  [scale=.75, fill=black]{};
    \vertex (L4) at (10, -8.75)  [scale=.75, fill=black, label=above:$$]{};
    \vertex (L5) at (10.5, -7.75)  [scale=.75, fill=black, label=right:$$]{};
    \vertex (L6) at (10.5, -6.75)  [scale=.75, fill=black, label=below:$$]{};
    \vertex (L7) at (10.5, -5.75)  [scale=.75, fill=black, label=below:$$]{};
    
                \vertex (M1) at (13,-6.25)  [scale=.75, fill=black, label=below:$$]{};
    \vertex (M2) at (14,-5.75)  [scale=.75, fill=black, label=left:$$]{};
    \vertex (M3) at (15, -6.25)  [scale=.75, fill=black]{};
    \vertex (M4) at (14, -6.75)  [scale=.75, fill=black, label=above:$$]{};
    \vertex (M5) at (13, -7.75)  [scale=.75, fill=black, label=right:$$]{};
    \vertex (M6) at (15, -7.75)  [scale=.75, fill=black, label=below:$$]{};
    \vertex (M7) at (14, -8.25)  [scale=.75, fill=black, label=below:$$]{};
    
   \vertex (N1) at (0.5,-11.5)  [scale=.75, fill=black, label=below:$$]{};
    \vertex (N2) at (0.5,-10.5)  [scale=.75, fill=black, label=left:$$]{};
    \vertex (N3) at (1.5, -11)  [scale=.75, fill=black]{};
    \vertex (N4) at (1.5, -10)  [scale=.75, fill=black, label=above:$$]{};
    \vertex (N5) at (2.5, -10.5)  [scale=.75, fill=black, label=right:$$]{};
    \vertex (N6) at (2.5, -11.5)  [scale=.75, fill=black, label=below:$$]{};
    \vertex (N7) at (1.5, -12)  [scale=.75, fill=black, label=below:$$]{};
    
       \vertex (O1) at (5,-11.5)  [scale=.75, fill=black, label=below:$$]{};
    \vertex (O2) at (5,-10.5)  [scale=.75, fill=black, label=left:$$]{};
    \vertex (O3) at (6, -10.5)  [scale=.75, fill=black]{};
    \vertex (O4) at (6, -11.5)  [scale=.75, fill=black, label=above:$$]{};
    \vertex (O5) at (7, -11.5)  [scale=.75, fill=black, label=right:$$]{};
    \vertex (O6) at (8, -11.5)  [scale=.75, fill=black, label=below:$$]{};
    \vertex (O7) at (9, -11.5)  [scale=.75, fill=black, label=below:$$]{};

    \node(A) at (1, .75) []{$W_5$};
    \node(B) at (4.5, .75)[]{$W_6$};
    \node(C) at (8.5, .75)[]{$W_7$};
     \node(D) at (12, .75)[]{$W_8$};
      \node(E) at (15.5, .75)[]{$W_9$};
       \node(F) at (2, -4.5)[]{$W_{10}$};
         \node(G) at (6.5, -4.5)[]{$W_{11}$};
           \node(H) at (10.5, -4.5)[]{$W_{12}$};
             \node(I) at (14.5, -4.5)[]{$W_{13}$};
                \node(J) at (1.5, -9.)[]{$W_{14}$};    
          \node(K) at (6, -9.5)[]{$W_{15}$};
            \node(L) at (10, -9.5)[]{$W_{16}$};
               \node(M) at (14, -9.5)[]{$W_{17}$};
               
              \node(N) at (1.5, -13)[]{$W_{18}$};
               \node(O) at (6.5, -13)[]{$W_{19}$};

    \path 
    	(01) edge (02)
	(02) edge (03)
	(03) edge (04)
	(04) edge (05)
	(05) edge (06)
	(06) edge (07)
	(01) edge (07)
	(02) edge (04)
	(03) edge (05)
	
	(0B1) edge (0B2)
	(0B2) edge (0B3)
	(0B3) edge (0B4)
	(0B4) edge (0B5)
	(0B5) edge (0B6)
	(0B6) edge (0B7)
	(0B1) edge (0B7)
	(0B2) edge (0B4)
	(0B3) edge (0B5)
	(0B3) edge (0B7)
	
  	(0C1) edge (0C2)
	(0C2) edge (0C3)
	(0C3) edge (0C4)
	(0C4) edge (0C5)
	(0C5) edge (0C6)
	(0C6) edge (0C7)
	(0C1) edge (0C7)
	(0C2) edge (0C4)
	(0C3) edge (0C5)  
	(0C4) edge (0C6)
	
	(0D1) edge (0D2)
	(0D2) edge [bend right] (0D6)
	(0D2) edge  (0D7)
	(0D2) edge (0D3)
	(0D3) edge (0D4)
	(0D4) edge (0D5)
	(0D2) edge (0D4)
	(0D3) edge (0D5)  
	
	(1) edge (2)
	(1) edge (3)
	(2) edge (3)
	(4) edge (3)
	(4) edge (5)
	(4) edge (6)
	(4) edge (7)
	
	(B1) edge (B2)
	(B1) edge (B3)
	(B2) edge (B3)
	(B3) edge (B4)
	(B3) edge (B5)
	(B4) edge (B6)
	(B5) edge (B7)

	(C1) edge (C2)
  	(C1) edge (C3)
	(C2) edge (C4)
	(C3) edge (C4)
	(C4) edge (C5)
	(C3) edge (C5)
	(C5) edge (C6)
	(C6) edge (C7)
	
		(D1) edge (D2)
  	(D1) edge (D3)
	(D2) edge (D4)
	(D3) edge (D4)
	(D4) edge (D5)
	(D3) edge (D5)
	(D5) edge (D6)
	(D5) edge (D7)
	(D6) edge (D7)
	
		(E1) edge (E2)
  	(E1) edge (E3)
	(E2) edge (E4)
	(E3) edge (E4)
	(E4) edge (E5)
	(E5) edge (E6)
	(E5) edge (E7)
	(E6) edge (E7)
	
		(F1) edge (F2)
  	(F1) edge (F3)
	(F2) edge (F4)
	(F3) edge (F4)
	(F4) edge (F6)
		(F3) edge (F5)
	(F5) edge (F6)
	(F5) edge (F7)
	(F6) edge (F7)
	
	(G1) edge (G2)
  	(G1) edge (G3)
	(G2) edge (G3)
	(G3) edge (G4)
	(G3) edge (G5)
		(G4) edge (G5)
	(G3) edge (G6)
	(G6) edge (G7)

	(H1) edge (H2)
  	(H1) edge (H3)
	(H2) edge (H3)
	(H3) edge (H4)
	(H4) edge (H5)
		(H5) edge (H6)
		(H1) edge (H5)
  	(H2) edge (H6)
	(H4) edge (H3)
	(H7) edge (H4)
	(H6) edge (H7)
	
	(I1) edge (I2)
  	(I1) edge (I3)
	(I2) edge (I3)
	(I3) edge (I4)
	(I4) edge (I5)
		(I5) edge (I3)
  	(I3) edge (I6)
	(I3) edge (I7)
	(I6) edge (I7)
	
	(J1) edge (J2)
	(J2) edge (J3)
	(J3) edge (J4)
	(J4) edge (J5)
	(J7) edge (J6)
	(J1) edge (J7)
	(J5) edge (J6)
	(J2) edge (J4)
	(J5) edge (J2)
	(J5) edge (J3)

	(K1) edge (K2)
	(K2) edge (K3)
	(K3) edge (K4)
	(K1) edge (K4)
	(K1) edge (K3)
	(K2) edge (K4)
	(K1) edge (K5)
	(K3) edge (K6)
	(K5) edge (K6)
	(K6) edge (K7)
	(K5) edge (K7)
	(K4) edge (K7)
	
	(L1) edge (L2)
	(L2) edge (L3)
	(L3) edge (L4)
	(L4) edge (L5)
	(L5) edge (L6)
	(L6) edge (L7)
	(L1) edge (L7)
	(L1) edge (L6)
	(L2) edge (L6)
	(L2) edge (L7)
	(L3) edge (L5)
	
		(M1) edge (M2)
	(M2) edge (M3)
	(M3) edge (M4)
	(M1) edge (M4)
	(M1) edge (M3)
	(M2) edge (M4)
	(M5) edge (M6)
	(M6) edge (M7)
	(M5) edge (M7)
	(M4) edge (M7)
	
	(N5) edge (N2)
	(N2) edge (N3)
	(N3) edge (N4)
	(N5) edge (N4)
	(N2) edge (N4)
	(N5) edge (N3)
	(N3) edge (N1)
	(N3) edge (N6)
	(N3) edge (N7)
	
	(O1) edge (O2)
	(O2) edge (O3)
	(O3) edge (O4)
	(O4) edge (O1)
	(O2) edge (O4)
	(O1) edge (O3)
	(O4) edge (O5)
	(O5) edge (O6)
	(O6) edge (O7)

    ;

\end{tikzpicture}
\caption{All connected well-edge dominated graphs containing a triangle of order $7$ where only $W_1, W_2, W_3$, and $W_4$ contain an induced diamond, and only $W_{14}, W_{15}, W_{16}, W_{17}, W_{18}$, and $W_{19}$ contain a $K_4$\NewSarah{.}}
\label{ND7}
\end{center}
\end{figure}

\newpage
\subsection{Connected, well-edge-dominated graphs with $8$ vertices}

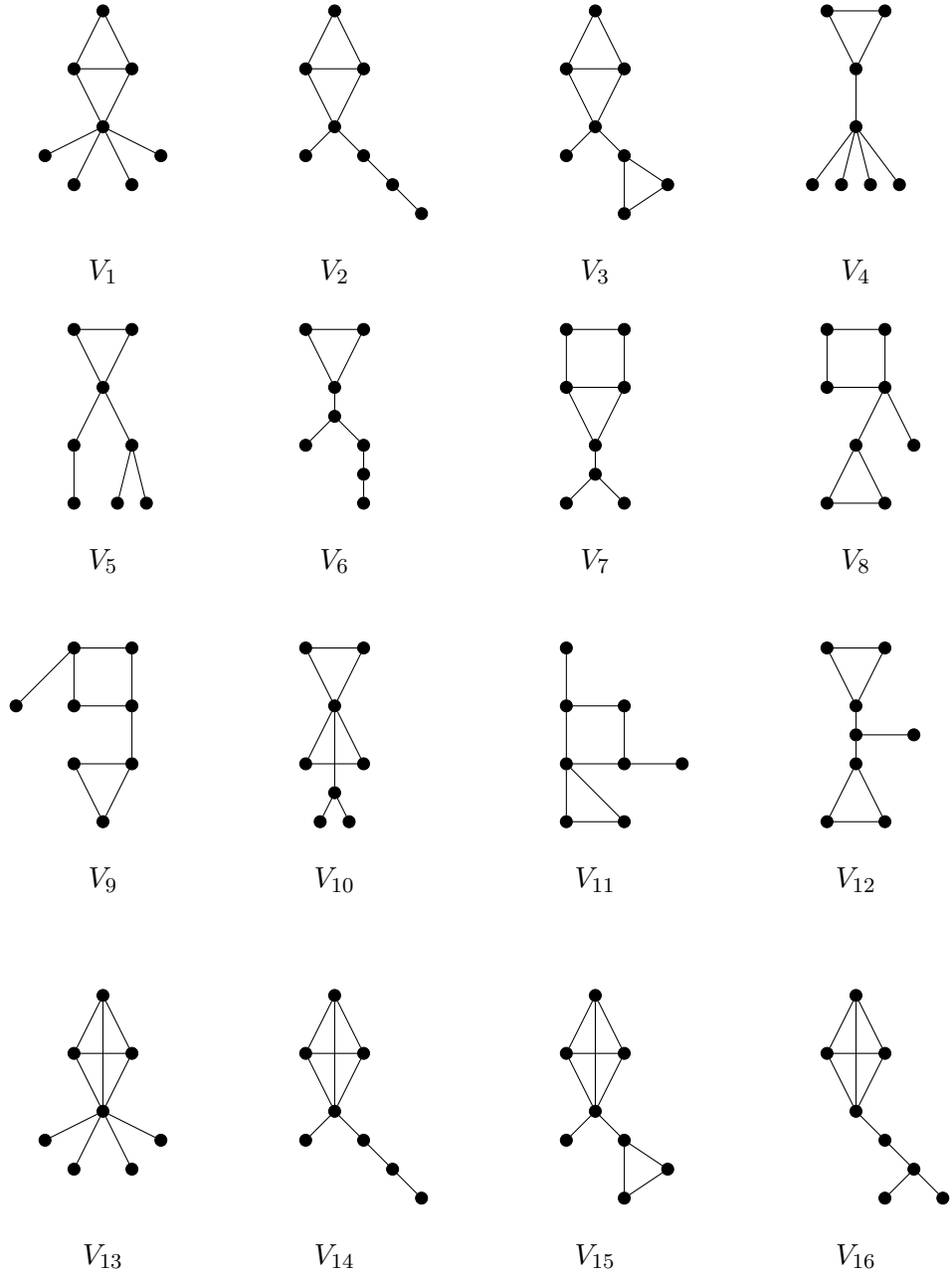
\begin{figure}[h!]
\begin{center}
\begin{tikzpicture}[scale=.77]

    \vertex (01) at (0,7)  [scale=.75, fill=black, label=below:]{};
    \vertex (02) at (.5,8.5)  [scale=.75, fill=black, label=left:]{};
    \vertex (03) at (1, 7.5)  [scale=.75, fill=black, label=right:]{};
    \vertex (04) at (1, 9.5)  [scale=.75, fill=black, label=above:]{};
    \vertex (05) at (1.5, 8.5)  [scale=.75, fill=black, label=right:]{};
    \vertex (06) at (1.5, 6.5)  [scale=.75, fill=black, label=below:]{};
    \vertex (07) at (.5, 6.5)  [scale=.75, fill=black, label=below:]{};
     \vertex (08) at (2, 7)  [scale=.75, fill=black, label=below:]{};
  
     \vertex (0B1) at (4.5,7)  [scale=.75, fill=black, label=below:]{};
    \vertex (0B2) at (4.5,8.5)  [scale=.75, fill=black, label=left:]{};
    \vertex (0B3) at (5, 7.5)  [scale=.75, fill=black, label=below:]{};
    \vertex (0B4) at (5, 9.5)  [scale=.75, fill=black, label=above:]{};
    \vertex (0B5) at (5.5, 8.5)  [scale=.75, fill=black, label=right:]{};
    \vertex (0B6) at (6.5, 6)  [scale=.75, fill=black, label=below:]{};
    \vertex (0B7) at (6, 6.5)  [scale=.75, fill=black, label=below:]{};
    \vertex (0B8) at (5.5, 7)  [scale=.75, fill=black, label=below:]{};

       \vertex (0C1) at (9,7)  [scale=.75, fill=black, label=below:]{};
    \vertex (0C2) at (9,8.5)  [scale=.75, fill=black, label=left:]{};
    \vertex (0C3) at (9.5, 7.5)  [scale=.75, fill=black, label=below:]{};
    \vertex (0C4) at (9.5, 9.5)  [scale=.75, fill=black, label=above:]{};
    \vertex (0C5) at (10, 8.5)  [scale=.75, fill=black, label=right:]{};
    \vertex (0C6) at (10, 6)  [scale=.75, fill=black, label=below:]{};
    \vertex (0C7) at (10.75, 6.5)  [scale=.75, fill=black, label=right:]{};
    \vertex (0C8) at (10, 7)  [scale=.75, fill=black, label=right:]{};
    
      \vertex (1) at (13.5,9.5)  [scale=.75, fill=black, label=above:]{};
    \vertex (2) at (14.5, 9.5)  [scale=.75, fill=black, label=above:]{};
    \vertex (3) at (14, 8.5)  [scale=.75, fill=black, label=right:]{};
    \vertex (4) at (14, 7.5)  [scale=.75, fill=black, label=right:]{};
    \vertex (5) at (13.25, 6.5)  [scale=.75, fill=black, label=below:]{};
    \vertex (6) at (13.75, 6.5)  [scale=.75, fill=black, label=below:]{};
    \vertex (7) at (14.25, 6.5)  [scale=.75, fill=black, label=below:]{};
     \vertex (8) at (14.75, 6.5)  [scale=.75, fill=black, label=below:]{};

    \node(A) at (1, 5) []{$V_1$};
    \node(B) at (5, 5)[]{$V_2$};
    \node(C) at (9.5, 5)[]{$V_3$};
    \node(A) at (14, 5) []{$V_4$};

     \vertex (B1) at (.5,4)  [scale=.75, fill=black, label=below:]{};
    \vertex (B2) at (1.5,4)  [scale=.75, fill=black, label=left:]{};
    \vertex (B3) at (1, 3)  [scale=.75, fill=black]{};
    \vertex (B4) at (.5, 2)  [scale=.75, fill=black, label=above:]{};
    \vertex (B5) at (1.5, 2)  [scale=.75, fill=black, label=right:]{};
    \vertex (B6) at (.5, 1)  [scale=.75, fill=black, label=below:]{};
    \vertex (B7) at (1.25, 1)  [scale=.75, fill=black, label=below:]{};
       \vertex (B8) at (1.75, 1)  [scale=.75, fill=black, label=below:]{};

        \vertex (C1) at (4.5,4)  [scale=.75, fill=black, label=below:]{};
    \vertex (C2) at (5.5,4)  [scale=.75, fill=black, label=left:]{};
    \vertex (C3) at (5, 3)  [scale=.75, fill=black]{};
    \vertex (C4) at (5, 2.5)  [scale=.75, fill=black, label=above:]{};
    \vertex (C5) at (4.5, 2)  [scale=.75, fill=black, label=right:]{};
    \vertex (C6) at (5.5, 2)  [scale=.75, fill=black, label=below:]{};
    \vertex (C7) at (5.5, 1.5)  [scale=.75, fill=black, label=below:]{};
        \vertex (C8) at (5.5, 1)  [scale=.75, fill=black, label=below:]{};

        \vertex (D1) at (9,4)  [scale=.75, fill=black, label=below:]{};
    \vertex (D2) at (10,4)  [scale=.75, fill=black, label=left:]{};
    \vertex (D3) at (9, 3)  [scale=.75, fill=black]{};
    \vertex (D4) at (10, 3)  [scale=.75, fill=black, label=above:]{};
    \vertex (D5) at (9.5, 2)  [scale=.75, fill=black, label=right:]{};
    \vertex (D8) at (9.5, 1.5)  [scale=.75, fill=black, label=below:]{};
    \vertex (D6) at (9, 1)  [scale=.75, fill=black, label=below:]{};
    \vertex (D7) at (10, 1)  [scale=.75, fill=black, label=below:]{};
    
     \vertex (E1) at (13.5,4)  [scale=.75, fill=black, label=below:]{};
    \vertex (E2) at (14.5,4)  [scale=.75, fill=black, label=left:]{};
    \vertex (E3) at (13.5, 3)  [scale=.75, fill=black]{};
    \vertex (E4) at (14.5, 3)  [scale=.75, fill=black, label=above:]{};
    \vertex (E5) at (14, 2)  [scale=.75, fill=black, label=right:]{};
    \vertex (E6) at (13.5, 1)  [scale=.75, fill=black, label=below:]{};
    \vertex (E7) at (14.5, 1)  [scale=.75, fill=black, label=below:]{};
      \vertex (E8) at (15,2)  [scale=.75, fill=black, label=below:]{};
    
         \vertex (F1) at (.5,-1.5)  [scale=.75, fill=black, label=below:]{};
    \vertex (F2) at (1.5,-1.5)  [scale=.75, fill=black, label=left:]{};
    \vertex (F3) at (.5, -2.5)  [scale=.75, fill=black]{};
    \vertex (F4) at (1.5, -2.5)  [scale=.75, fill=black, label=above:]{};
    \vertex (F5) at (.5, -3.5)  [scale=.75, fill=black, label=right:]{};
    \vertex (F6) at (1.5, -3.5)  [scale=.75, fill=black, label=below:]{};
    \vertex (F7) at (1, -4.5)  [scale=.75, fill=black, label=below:]{};
      \vertex (F8) at (-.5, -2.5)  [scale=.75, fill=black, label=below:]{};

         \vertex (G1) at (4.5,-1.5)  [scale=.75, fill=black, label=below:]{};
    \vertex (G2) at (5.5,-1.5)  [scale=.75, fill=black, label=left:]{};
    \vertex (G3) at (5, -2.5)  [scale=.75, fill=black]{};
    \vertex (G4) at (4.5, -3.5)  [scale=.75, fill=black, label=above:]{};
    \vertex (G5) at (5.5, -3.5)  [scale=.75, fill=black, label=right:]{};
    \vertex (G6) at (5, -4)  [scale=.75, fill=black, label=below:]{};
    \vertex (G7) at (4.75, -4.5)  [scale=.75, fill=black, label=below:]{};
        \vertex (G8) at (5.25, -4.5)  [scale=.75, fill=black, label=below:]{};

         \vertex (H1) at (9,-1.5)  [scale=.75, fill=black, label=below:]{};
    \vertex (H2) at (10,-2.5)  [scale=.75, fill=black, label=left:]{};
    \vertex (H3) at (9, -2.5)  [scale=.75, fill=black]{};
    \vertex (H4) at (10, -3.5)  [scale=.75, fill=black, label=above:]{};
    \vertex (H5) at (9, -3.5)  [scale=.75, fill=black, label=right:]{};
    \vertex (H6) at (10, -4.5)  [scale=.75, fill=black, label=below:]{};
    \vertex (H7) at (9, -4.5)  [scale=.75, fill=black, label=below:]{};
       \vertex (H8) at (11, -3.5)  [scale=.75, fill=black, label=below:]{};
    
      \vertex (I1) at (13.5,-1.5)  [scale=.75, fill=black, label=below:]{};
    \vertex (I2) at (14.5,-1.5)  [scale=.75, fill=black, label=left:]{};
    \vertex (I3) at (14, -2.5)  [scale=.75, fill=black]{};
    \vertex (I4) at (14, -3)  [scale=.75, fill=black, label=above:]{};
    \vertex (I5) at (14, -3.5)  [scale=.75, fill=black, label=right:]{};
    \vertex (I6) at (13.5, -4.5)  [scale=.75, fill=black, label=below:]{};
    \vertex (I7) at (14.5, -4.5)  [scale=.75, fill=black, label=below:]{};
        \vertex (I8) at (15, -3)  [scale=.75, fill=black, label=below:]{};
        
            \vertex (J1) at (0,-10)  [scale=.75, fill=black, label=below:]{};
    \vertex (J2) at (.5,-8.5)  [scale=.75, fill=black, label=left:]{};
    \vertex (J3) at (1, -9.5)  [scale=.75, fill=black, label=right:]{};
    \vertex (J4) at (1, -7.5)  [scale=.75, fill=black, label=above:]{};
    \vertex (J5) at (1.5, -8.5)  [scale=.75, fill=black, label=right:]{};
    \vertex (J6) at (1.5, -10.5)  [scale=.75, fill=black, label=below:]{};
    \vertex (J7) at (.5, -10.5)  [scale=.75, fill=black, label=below:]{};
     \vertex (J8) at (2, -10)  [scale=.75, fill=black, label=below:]{};
  
     \vertex (K1) at (4.5,-10)  [scale=.75, fill=black, label=below:]{};
    \vertex (K2) at (4.5,-8.5)  [scale=.75, fill=black, label=left:]{};
    \vertex (K3) at (5, -9.5)  [scale=.75, fill=black, label=below:]{};
    \vertex (K4) at (5, -7.5)  [scale=.75, fill=black, label=above:]{};
    \vertex (K5) at (5.5, -8.5)  [scale=.75, fill=black, label=right:]{};
    \vertex (K6) at (6.5, -11)  [scale=.75, fill=black, label=below:]{};
    \vertex (K7) at (6, -10.5)  [scale=.75, fill=black, label=below:]{};
    \vertex (K8) at (5.5, -10)  [scale=.75, fill=black, label=below:]{};

       \vertex (L1) at (9,-10)  [scale=.75, fill=black, label=below:]{};
    \vertex (L2) at (9,-8.5)  [scale=.75, fill=black, label=left:]{};
    \vertex (L3) at (9.5, -9.5)  [scale=.75, fill=black, label=below:]{};
    \vertex (L4) at (9.5, -7.5)  [scale=.75, fill=black, label=above:]{};
    \vertex (L5) at (10, -8.5)  [scale=.75, fill=black, label=right:]{};
    \vertex (L6) at (10, -11)  [scale=.75, fill=black, label=below:]{};
    \vertex (L7) at (10.75, -10.5)  [scale=.75, fill=black, label=right:]{};
    \vertex (L8) at (10, -10)  [scale=.75, fill=black, label=right:]{};

     \vertex (M1) at (14.5,-11)  [scale=.75, fill=black, label=below:]{};
    \vertex (M2) at (13.5,-8.5)  [scale=.75, fill=black, label=left:]{};
    \vertex (M3) at (14, -9.5)  [scale=.75, fill=black, label=below:]{};
    \vertex (M4) at (14, -7.5)  [scale=.75, fill=black, label=above:]{};
    \vertex (M5) at (14.5, -8.5)  [scale=.75, fill=black, label=right:]{};
    \vertex (M6) at (15.5, -11)  [scale=.75, fill=black, label=below:]{};
    \vertex (M7) at (15, -10.5)  [scale=.75, fill=black, label=below:]{};
    \vertex (M8) at (14.5, -10)  [scale=.75, fill=black, label=below:]{};

    \node(B) at (1, 0)[]{$V_5$};
    \node(C) at (5, 0)[]{$V_6$};
     \node(D) at (9.5, 0)[]{$V_7$};
      \node(E) at (14, 0)[]{$V_8$};
       \node(F) at (1, -5.5)[]{$V_9$};
         \node(G) at (5, -5.5)[]{$V_{10}$};
           \node(H) at (9.5, -5.5)[]{$V_{11}$};
             \node(I) at (14, -5.5)[]{$V_{12}$};
               \node(J) at (1, -12)[]{$V_{13}$};
                 \node(K) at (5, -12)[]{$V_{14}$};
                   \node(L) at (9.5, -12)[]{$V_{15}$};
                     \node(M) at (14, -12)[]{$V_{16}$};

    \path 
    
    	(01) edge (03)
	(06) edge (03)
	(07) edge (03)
	(08) edge (03)
	(02) edge (03)
	(04) edge (05)
	(02) edge (04)
	(03) edge (05)
	(02) edge(05)
	
	(0B1) edge (0B3)
	(0B2) edge (0B3)
	(0B2) edge (0B5)
	(0B4) edge (0B5)
	(0B3) edge (0B5)
	(0B4) edge (0B2)
	(0B3) edge (0B8)
	(0B8) edge (0B7)
	(0B6) edge (0B7)

  	(0C1) edge (0C3)
	(0C2) edge (0C3)
	(0C2) edge (0C5)
	(0C4) edge (0C5)
	(0C3) edge (0C5)
	(0C4) edge (0C2)
	(0C3) edge (0C8)
	(0C8) edge (0C7)
	(0C6) edge (0C7)
	(0C8) edge (0C6)
	(1) edge (2)
	(1) edge (3)
	(2) edge (3)
	(4) edge (3)
	(4) edge (5)
	(4) edge (6)
	(4) edge (7)
	(4) edge (8)
	
	(B1) edge (B2)
	(B1) edge (B3)
	(B2) edge (B3)
	(B3) edge (B4)
	(B3) edge (B5)
	(B4) edge (B6)
	(B5) edge (B7)
	(B5) edge (B8)

	(C1) edge (C2)
  	(C1) edge (C3)
	(C2) edge (C3)
	(C3) edge (C4)
	(C4) edge (C5)
	(C4) edge (C6)
	(C6) edge (C7)
	(C7) edge (C8)

		(D1) edge (D2)
  	(D1) edge (D3)
	(D2) edge (D4)
	(D3) edge (D4)
	(D4) edge (D5)
	(D3) edge (D5)
	(D5) edge (D8)
	(D8) edge (D7)
	(D8) edge (D6)

		(E1) edge (E2)
  	(E1) edge (E3)
	(E2) edge (E4)
	(E3) edge (E4)
	(E4) edge (E5)
	(E4) edge (E8)
	(E5) edge (E6)
	(E5) edge (E7)
	(E6) edge (E7)
	
		(F1) edge (F2)
  	(F1) edge (F3)
	(F2) edge (F4)
	(F3) edge (F4)
	(F4) edge (F6)
		(F1) edge (F8)
	(F5) edge (F6)
	(F5) edge (F7)
	(F6) edge (F7)
	
	(G1) edge (G2)
  	(G1) edge (G3)
	(G2) edge (G3)
	(G3) edge (G4)
	(G3) edge (G5)
		(G4) edge (G5)
	(G3) edge (G6)
	(G6) edge (G7)
	(G6) edge (G8)

(H3) edge (H2)
(H4) edge (H2)
	(H5) edge (H3)
  	(H5) edge (H4)
	(H5) edge (H6)
	(H7) edge (H5)
	(H6) edge (H7)
	(H3) edge (H1)
	(H4) edge (H8)
	
	(I1) edge (I2)
  	(I1) edge (I3)
	(I2) edge (I3)
	
  	(I5) edge (I6)
	(I5) edge (I7)
	(I6) edge (I7)
	(I4) edge (I8)
	(I3) edge (I4)
	(I5) edge (I4)
	
	  (J1) edge (J3)
	(J6) edge (J3)
	(J7) edge (J3)
	(J8) edge (J3)
	(J2) edge (J3)
	(J4) edge (J5)
	(J2) edge (J4)
	(J3) edge (J5)
	(J2) edge(J5)
	(J4) edge (J3)
	
	(K1) edge (K3)
	(K2) edge (K3)
	(K2) edge (K5)
	(K4) edge (K5)
	(K3) edge (K5)
	(K4) edge (K2)
	(K3) edge (K8)
	(K8) edge (K7)
	(K6) edge (K7)
	(K4) edge (K3)

  	(L1) edge (L3)
	(L2) edge (L3)
	(L2) edge (L5)
	(L4) edge (L5)
	(L3) edge (L5)
	(L4) edge (L2)
	(L3) edge (L8)
	(L8) edge (L7)
	(L6) edge (L7)
	(L8) edge (L6)
	(L4) edge (L3)
	
		(M1) edge (M7)
	(M2) edge (M3)
	(M2) edge (M5)
	(M4) edge (M5)
	(M3) edge (M5)
	(M4) edge (M2)
	(M3) edge (M8)
	(M8) edge (M7)
	(M6) edge (M7)
	(M4) edge (M3)

    ;

\end{tikzpicture}
\caption{All connected well-edge dominated graphs containing a triangle of order $8$ where only $V_1, V_2$, and $V_3$  contain an induced diamond, and only $V_{13}, V_{14}, V_{15}$, and $V_{16}$ contain a $K_4$\NewSarah{.}}
\label{ND8}
\end{center}
\end{figure}

\end{document}